\documentclass[11pt]{amsart}%
\usepackage{amscd,amsmath,latexsym,amsthm,amsfonts,amssymb,graphicx,color,geometry}
\usepackage{enumerate}
\usepackage[pdftex,plainpages=false,pdfpagelabels,
colorlinks=true,citecolor=blue,linkcolor=black,urlcolor=black,
filecolor=black,bookmarksopen=true,unicode=true]{hyperref}
\usepackage[norefs,nocites]{refcheck}
\usepackage{amsmath}
\usepackage{amsfonts}
\usepackage{amssymb}
\usepackage[T1]{fontenc}
\usepackage{graphicx}
\usepackage{xcolor}
\usepackage{colortbl}
\usepackage{float}
\usepackage{framed}
\usepackage{pgf,tikz,pgfplots}
\usepackage{tkz-base}
\usepackage{multicol}
\usepackage{tkz-fct}%
\providecommand{\U}[1]{\protect\rule{.1in}{.1in}}
\tikzset{>=Triangle}
\newtheorem{theorem}{Theorem}[section]
\newtheorem{proposition}[theorem]{Proposition}

\newtheorem{corollary}[theorem]{Corollary}

\newtheorem{remark}[theorem]{Remark}

\newtheorem{lemma}[theorem]{Lemma}
\newtheorem{problem}[theorem]{Problem}

\numberwithin{equation}{section}
\pgfplotsset{compat=1.17}
\begin{document}
\title[Regularity of the coefficients of multilinear forms on sequence spaces]{Regularity of the coefficients of multilinear forms on sequence spaces}
\author[D. M. Pellegrino]{Daniel M. Pellegrino}
\address{Departamento de Matem\'{a}tica \\
Universidade Federal da Para\'{\i}ba \\
58.051-900 - Jo\~{a}o Pessoa, Brazil.}
\email{daniel.pellegrino@academico.ufpb.br}
\author[A. Raposo Jr.]{Anselmo Raposo Jr.}
\address{Departamento de Matem\'{a}tica \\
Universidade Federal do Maranh\~{a}o \\
65085-580 - S\~{a}o Lu\'{\i}s, Brazil.}
\email{anselmo.junior@ufma.br}
\author[D.M. Serrano-Rodr\'{\i}guez]{Diana M. Serrano-Rodr\'{\i}guez}
\address{Departamento de Matem\'{a}ticas \\
Universidad Nacional de Colombia \\
111321 - Bogot\'{a}, Colombia.}
\email{Diserranor@unal.edu.co}
\thanks{D. Pellegrino is supported by CNPq Grant 307327/2017-5 and Grant 2019/0014
Paraiba State Research Foundation (FAPESQ) }

\begin{abstract}
The investigation of regularity/summability properties of the coefficients of
bilinear forms in sequence spaces was initiated by Littlewood in $1930$.
Nowadays, this topic has important connections with other fields of Pure and
Applied Mathematics as Complex Analysis, Quantum Information Theory,
Theoretical Computer Science and Combinatorial Games. In this paper we explore
a regularity technique to obtain optimal parameters for several results in
this framework, extending/generalizing theorems of Osikiewicz and Tonge
($2001$), Albuquerque \textit{et al.} ($2016$), Aron \textit{et al.} ($2017$),
Albuquerque and Rezende ($2018$), Paulino ($2020$), among others.

\end{abstract}
\maketitle
\tableofcontents

\section{Introduction}

The investigation of summability properties of the coefficients of multilinear
forms defined on sequence spaces has its origins in 1930 with Littlewood's
seminal paper \cite{l1930}. Littlewood proved that for every bilinear form
$A\colon\mathbb{C}^{n}\times\mathbb{C}^{n}\rightarrow\mathbb{C}$, we have%
\begin{equation}
\left(
%TCIMACRO{\tsum \limits_{i,j=1}^{\infty}}%
%BeginExpansion
{\textstyle\sum\limits_{i,j=1}^{\infty}}
%EndExpansion
\left\vert A\left(  e_{i},e_{j}\right)  \right\vert ^{4/3}\right)  ^{3/4}%
\leq\sqrt{2}\sup\left\{  \left\vert A(z^{(1)},z^{(2)})\right\vert
:z^{(1)},z^{(2)}\in\mathbb{D}^{n}\right\}  \text{,} \label{q12}%
\end{equation}
where $\mathbb{D}^{n}$ represents the open unit polydisc in $\mathbb{C}^{n}$.
Moreover, the exponent $4/3$ cannot be improved in the sense that, if we
replace $4/3$ by a smaller exponent, it is not possible to change $\sqrt{2}$
by a constant not depending on $n$. This result is known as Littlewood's $4/3$
inequality and its original motivation was a problem posed by P.J. Daniell
concerning functions of bounded variation. Note that the terms $A\left(
e_{i},e_{j}\right)  $ are precisely the coefficients of the bilinear form $A$.
Nowadays it is well-known that summability properties of the coefficients of
multilinear forms play an important role in Mathematics and related fields.
For instance, in $1931$ Bohnenblust and Hille \cite{BH} extended Littlewood's
inequality to multilinear forms in order to investigate Bohr's absolute
convergence problem, i.e., to determine the maximal width $T$ of the vertical
strip in which a Dirichlet series $%
%TCIMACRO{\tsum \nolimits_{n=1}^{\infty}}%
%BeginExpansion
{\textstyle\sum\nolimits_{n=1}^{\infty}}
%EndExpansion
a_{n}n^{-s}$ converges uniformly but not absolutely. In $2011$, Defant,
Frerick, Ortega-Cerd\`{a}, Ouna\"{\i}es and Seip \cite{annals2011} revisited
the paper of Bohnenblust and Hille and proved that the constant of the
Bohnenblust--Hille inequality for homogeneous polynomials was
hypercontractive, obtaining important applications in Analytic Number Theory
and Complex Analysis. In 2014, Bayart, Pellegrino and Seoane \cite{BPSS}
showed that the constants of the polynomial Bohnenblust--Hille inequality were
in fact sub-exponential and, as a consequence, concluded that the exact
asymptotic growth of the Bohr radius of the $n$-dimensional polydisk is
$\sqrt{\log n/n}$. Applications of this bulk of results transcend the scope of
Pure Mathematics and can be found in Quantum Information Theory, Theoretical
Computer Science and Combinatorial Games (\cite{entropy,montanaro,
JFAjoguinho}).

The natural extension of Littlewood's $4/3$ inequality to multilinear forms in
$\mathbb{K}^{n}$ ($\mathbb{K}=\mathbb{R}$ or $\mathbb{C}$) spaces, with
$\ell_{p}$ norms, was initiated by Hardy and Littlewood in $1934$ for bilinear
forms (\cite{hardy}) and extended to multilinear forms by Praciano-Pereira
\cite{pra} in $1981$; since then several authors have investigated this
subject (see \cite{b,bayartEUROPEAN,dimant,998800} and the references
therein). In this new context, for an $m$-linear form $A\colon\mathbb{K}%
^{n}\times\cdots\times\mathbb{K}^{n}\rightarrow\mathbb{K}$, the supremum at
the right-hand-side of (\ref{q12}) is replaced by%
\[
\left\Vert A\right\Vert :=\sup\left\{  \left\vert A(z^{(1)},\ldots
,z^{(m)})\right\vert :\left\Vert z^{(j)}\right\Vert _{\ell_{p_{j}}^{n}}%
\leq1\right\}  \text{.}%
\]
Above and henceforth, for the sake of simplicity, $\ell_{p}^{n}$ denotes
$\mathbb{K}^{n}$ with the $\ell_{p}$ norm. Even if we do not explicitly
mention, all inequalities along this paper hold for all positive integers $n$
and the respective constants do not depend on $n$; when we write $C_{m}$ it
means that the constant just depend on $m$ (the degree of $m$-linearity of a
multilinear form and we shall assume that $m\geq2$). As usual, $p^{\ast}$
shall denote the conjugate of $p$, i.e., $1/p+1/p^{\ast}=1$ and we assume that
$1/\infty=0$ and $1/0=\infty$.

There is a vast recent literature related to the Hardy--Littlewood
inequalities (HL for short) for multilinear forms in sequence spaces. The main
lines of research are the search of optimal exponents and optimal constants
(see \cite{ap, b, pt, 998800} and the references therein). Following the lines
of the seminal paper of Hardy and Littlewood, the investigation is usually
divided in two cases:

\begin{itemize}
\item $1/p_{1}+\cdots+1/p_{m}\leq1/2$,

\item $1/2\leq1/p_{1}+\cdots+1/p_{m}<1$.
\end{itemize}

In the first case, Praciano-Pereira \cite{pra} extended the bilinear result of
Hardy and Littlewood by proving that there exists a constant $C_{m}$ such
that
\begin{equation}
\left(  \sum_{j_{1}=1}^{n}\cdots\sum_{j_{m}=1}^{n}\left\vert A(e_{j_{1}%
},\ldots,e_{j_{m}})\right\vert ^{\mu}\right)  ^{\frac{1}{\mu}}\leq
C_{m}\left\Vert A\right\Vert \text{,} \label{i99}%
\end{equation}
for all $m$-linear forms $A\colon\ell_{p_{1}}^{n}\times\cdots\times\ell
_{p_{m}}^{n}\rightarrow\mathbb{K}$, where
\[
\mu:=\frac{2m}{m+1-2\left(  1/p_{1}+\cdots+1/p_{m}\right)  }\text{.}%
\]
This result was recently extended in \cite[Theorem 2.2]{alb2}. In fact, if
$1/p_{1}+\cdots+1/p_{m}\leq1/2$, choosing $\left(  r,q\right)  =\left(
1,2\right)  $ and $Y=\mathbb{K}$ in \cite[Theorem 2.2]{alb2} we have: if
$t_{1},\dots,t_{m}\in\left(  \left[  1-\left(  1/p_{1}+\cdots+1/p_{m}\right)
\right]  ^{-1},2\right)  $, there is a constant $C_{m}$ such that
\begin{equation}
\left(  \sum\limits_{j_{1}=1}^{n}\left(  \cdots\left(  \sum\limits_{j_{m}%
=1}^{n}\left\vert A\left(  e_{j_{1}},\ldots,e_{j_{m}}\right)  \right\vert
^{t_{m}}\right)  ^{\frac{t_{m-1}}{t_{m}}}\cdots\right)  ^{\frac{t_{1}}{t_{2}}%
}\right)  ^{\frac{1}{t_{1}}}\leq C_{m}\left\Vert A\right\Vert \text{,}
\label{aniso}%
\end{equation}
for all $m$-linear forms $A\colon\ell_{p_{1}}^{n}\times\cdots\times\ell
_{p_{m}}^{n}\rightarrow\mathbb{K}$ if, and only if,
\begin{equation}
\frac{1}{t_{1}}+\cdots+\frac{1}{t_{m}}\leq\frac{m+1}{2}-\left(  \dfrac
{1}{p_{1}}+\cdots+\dfrac{1}{p_{m}}\right)  \text{.} \label{211}%
\end{equation}

Notice that, due to the monotonicity of the $\ell_{p}$ norms, the interesting
case in (\ref{211}) is when the equality holds.

The equivalence $(\ref{aniso})\Leftrightarrow(\ref{211})$ shows that, in
general, there is no unique solution to the question: what is the optimal
exponent at the $j$-th position of the HL inequality? In fact, at least on the
aforementioned case, note that the optimal value of $t_{1}$ depends on
$t_{2},\ldots,t_{m}$ and so on. This motivates a more involved notion of
optimality, introduced in \cite[Definition 7.1]{PSRT}: an $m$-tuple of
exponents $\left(  t_{1},\ldots,t_{m}\right)  $ is called \textit{globally
sharp} if a Hardy-Littlewood type inequality holds for these exponents and,
for any $\varepsilon_{j}>0$ and $j=1,\ldots,m$, there is no HL inequality for
the $m$-tuple of exponents $\left(  t_{1},\ldots,t_{j-1},t_{j}-\varepsilon
_{j},t_{j+1},\ldots,t_{m}\right)  $.

From a different viewpoint a globally sharp $m$-tuple of exponents $\left(
t_{1},\ldots,t_{m}\right)  $ is a border point of the set of all admissible
exponents of a certain type of HL inequality.%

%TCIMACRO{\TeXButton{%
%\begin{figure}%
%[H]}{\begin{figure}[H]
%\centering\caption
%{A globally sharp exponent $(a,b,c)$ for $m=3$ in a certain HL inequality}}}%
%BeginExpansion
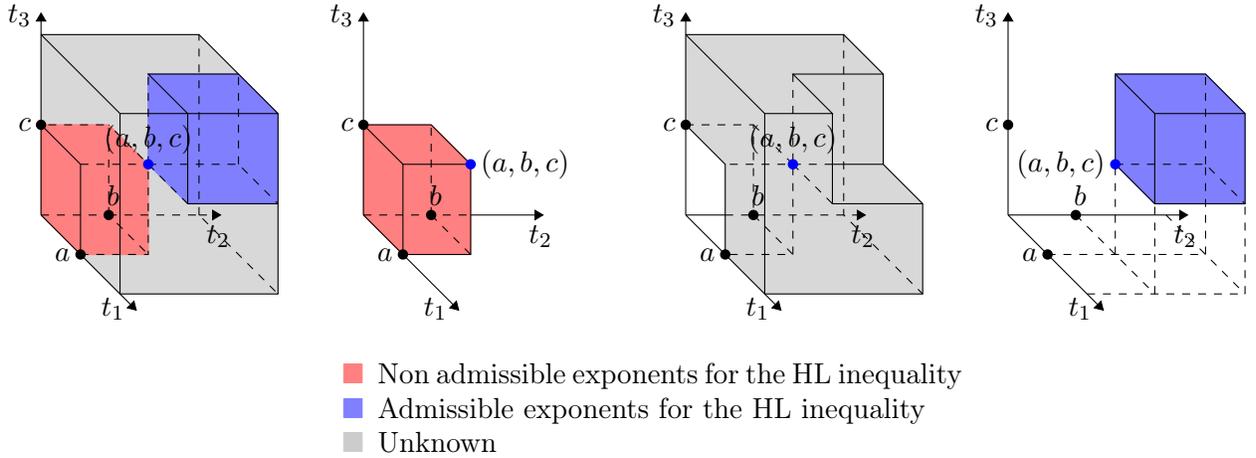
\begin{figure}[H]
\centering\caption
{A globally sharp exponent $(a,b,c)$ for $m=3$ in a certain HL inequality}%
%EndExpansion
%

%TCIMACRO{\TeXButton{%
%\begin{multicols}%
%{4}}{\begin{multicols}{4}}}%
%BeginExpansion
\begin{multicols}{4}%
\begin{tikzpicture}[x=0.3cm,y=0.3cm]
\draw
[line width=0.5pt,color=gray!40,fill=gray!40,opacity=0.75] (4.75,2.25)--(8.75,2.25)--(8.75,6.25)--(4.75,6.25)--(4.75,2.25);
\draw
[line width=0.5pt,color=gray!40,fill=gray!40,opacity=0.75] (0,4)--(0,8)--(7,8)--(8.75,6.25)--
(4.75,6.25)--(6.5,4.5)--(6.5,0.5)--(10.5,0.5)--(10.5,-3.5)--(3.5,-3.5)--(1.75,-1.75)--(1.75,2.25)
--(0,4);
\draw
[line width=0.5pt,color=red!50,fill=red!50] (0,0)--(0,4)--(3,4)--(4.75,2.25)--(4.75,-1.75)--
(1.75,-1.75)--(0,0);
\draw
[line width=0.5pt,color=blue!50,fill=blue!50] (4.75,2.25)--(4.75,6.25)--(8.75,6.25)--
(10.5,4.5)--(10.5,0.5)--(6.5,0.5)--(4.75,2.25);
\draw[dashed] (0,0)--(3,0)--(4.75,-1.75);
\draw[dashed] (3,0)--(3,4)--(0,4);
\draw[dashed] (3,0)--(7,0)--(10.5,-3.5);
\draw[dashed] (1.75,-1.75)--(4.75,-1.75)--(4.75,2.25);
\draw[dashed] (4.75,2.25)--(8.75,2.25)--(10.5,0.5);
\draw[dashed] (8.75,2.25)--(8.75,6.25);
\draw[dashed] (7,0)--(7,8);
\draw[dashed] (1.75,2.25)--(4.75,2.25);
\draw[dashed] (6.5,0.5)--(3,4);
\draw[dashed] (4.75,2.25)--(4.75,6.25);
\draw(1.75,-1.75)--(1.75,2.25);
\draw(3.5,-3.5)--(3.5,4.5)--(0,8);
\draw(3.5,4.5)--(10.5,4.5);
\draw(0,4)--(1.75,2.25);
\draw(0,8)--(7,8)--(10.5,4.5)--(10.5,-3.5)--(3.5,-3.5);
\draw(10.5,0.5)--(6.5,0.5)--(6.5,4.5)--(4.75,6.25)--(8.75,6.25);
\draw[->,dashed] (7,0)--(8,0);
\draw[->] (0,0)--(0,9);
\draw[->] (0,0)--(4.25,-4.25);
\draw[line width=0.5pt,color=black,fill=black] (0,4) circle (0.06cm);
\draw[line width=0.5pt,color=black,fill=black] (3,0) circle (0.06cm);
\draw[line width=0.5pt,color=black,fill=black] (1.75,-1.75) circle (0.06cm);
\draw[line width=0.5pt,color=blue,fill=blue] (4.75,2.25) circle (0.06cm);
\draw[color=black, left] (0,8.85) node {$t_{3}$};
\draw[color=black, below] (7.85,0) node {$t_{2}$};
\draw[color=black, left] (4.15,-4.15) node {$t_{1}$};
\draw[color=black, left] (1.75,-1.75) node {$a$};
\draw[color=black, above] (3.2,0) node {$b$};
\draw[color=black, left] (0,4) node {$c$};
\draw[color=black, above] (4.75,2.25) node {$(a,b,c)$};
\end{tikzpicture}%
%EndExpansion
%

%TCIMACRO{\TeXButton{Fig2}{\begin{tikzpicture}[x=0.3cm,y=0.3cm]
%\draw[color=white] (0,0)--(10.5,0);
%\draw
%[line width=0.5pt,color=red!50,fill=red!50] (0,0)--(0,4)--(3,4)--(4.75,2.25)--(4.75,-1.75)--(1.75,-1.75)--(0,0);
%\draw
%(0,4)--(3,4)--(4.75,2.25)--(4.75,-1.75)--(1.75,-1.75)--(1.75,2.25)--(0,4);
%\draw(1.75,2.25)--(4.75,2.25);
%\draw[dashed] (3,0)--(4.75,-1.75);
%\draw[dashed] (3,0)--(3,4);
%\draw[dashed] (0,0)--(4.75,0);
%\draw[->] (4.75,0)--(8,0);
%\draw[->] (0,0)--(0,9);
%\draw[->] (0,0)--(4.25,-4.25);
%\draw[line width=0.5pt,color=black,fill=black] (0,4) circle (0.06cm);
%\draw[line width=0.5pt,color=black,fill=black] (3,0) circle (0.06cm);
%\draw[line width=0.5pt,color=black,fill=black] (1.75,-1.75) circle (0.06cm);
%\draw[line width=0.5pt,color=blue,fill=blue] (4.75,2.25) circle (0.06cm);
%\draw[color=black, left] (0,8.85) node {$t_{3}$};
%\draw[color=black, below] (7.85,0) node {$t_{2}$};
%\draw[color=black, left] (4.15,-4.15) node {$t_{1}$};
%\draw[color=black, left] (1.75,-1.75) node {$a$};
%\draw[color=black, above] (3.2,0) node {$b$};
%\draw[color=black, left] (0,4) node {$c$};
%\draw[color=black, right] (4.75,2.25) node {$(a,b,c)$};
%\end{tikzpicture}}}%
%BeginExpansion
\begin{tikzpicture}[x=0.3cm,y=0.3cm]
\draw[color=white] (0,0)--(10.5,0);
\draw
[line width=0.5pt,color=red!50,fill=red!50] (0,0)--(0,4)--(3,4)--(4.75,2.25)--(4.75,-1.75)--(1.75,-1.75)--(0,0);
\draw
(0,4)--(3,4)--(4.75,2.25)--(4.75,-1.75)--(1.75,-1.75)--(1.75,2.25)--(0,4);
\draw(1.75,2.25)--(4.75,2.25);
\draw[dashed] (3,0)--(4.75,-1.75);
\draw[dashed] (3,0)--(3,4);
\draw[dashed] (0,0)--(4.75,0);
\draw[->] (4.75,0)--(8,0);
\draw[->] (0,0)--(0,9);
\draw[->] (0,0)--(4.25,-4.25);
\draw[line width=0.5pt,color=black,fill=black] (0,4) circle (0.06cm);
\draw[line width=0.5pt,color=black,fill=black] (3,0) circle (0.06cm);
\draw[line width=0.5pt,color=black,fill=black] (1.75,-1.75) circle (0.06cm);
\draw[line width=0.5pt,color=blue,fill=blue] (4.75,2.25) circle (0.06cm);
\draw[color=black, left] (0,8.85) node {$t_{3}$};
\draw[color=black, below] (7.85,0) node {$t_{2}$};
\draw[color=black, left] (4.15,-4.15) node {$t_{1}$};
\draw[color=black, left] (1.75,-1.75) node {$a$};
\draw[color=black, above] (3.2,0) node {$b$};
\draw[color=black, left] (0,4) node {$c$};
\draw[color=black, right] (4.75,2.25) node {$(a,b,c)$};
\end{tikzpicture}%
\begin{tikzpicture}[x=0.3cm,y=0.3cm]
\draw
[line width=0.5pt,color=gray!40,fill=gray!40,opacity=0.75] (0,4)--(0,8)--(7,8)--(8.75,6.25)--
(8.75,2.25)--(10.5,0.5)--(10.5,-3.5)--(3.5,-3.5)--(1.75,-1.75)--(1.75,2.25)--(0,4);
\draw
(0,8)--(7,8)--(8.75,6.25)--(8.75,2.25)--(10.5,0.5)--(10.5,-3.5)--(3.5,-3.5);
\draw(0,4)--(1.75,2.25)--(1.75,-1.75);
\draw(0,8)--(3.5,4.5)--(3.5,-3.5);
\draw(8.75,6.25)--(4.75,6.25)--(6.5,4.5)--(3.5,4.5);
\draw(10.5,0.5)--(6.5,0.5)--(6.5,4.5);
\draw(8.75,2.25)--(6.5,2.25);
\draw(0,0)--(1.75,0);
\draw[dashed] (0,4)--(3,4)--(3,0);
\draw[dashed] (3,4)--(4.75,2.25)--(1.75,2.25);
\draw[dashed] (3,0)--(4.75,-1.75)--(1.75,-1.75);
\draw[dashed] (4.75,-1.75)--(4.75,2.25);
\draw[dashed] (4.75,2.25)--(6.5,2.25);
\draw[dashed] (4.75,6.25)--(4.75,2.25);
\draw[dashed] (7,8)--(7,0)--(10.5,-3.5);
\draw[dashed] (4.75,2.25)--(6.5,0.5);
\draw[->,dashed] (1.75,0)--(8,0);
\draw[->] (0,0)--(0,9);
\draw[->] (0,0)--(4.25,-4.25);
\draw[line width=0.5pt,color=black,fill=black] (0,4) circle (0.06cm);
\draw[line width=0.5pt,color=black,fill=black] (3,0) circle (0.06cm);
\draw[line width=0.5pt,color=black,fill=black] (1.75,-1.75) circle (0.06cm);
\draw[line width=0.5pt,color=blue,fill=blue] (4.75,2.25) circle (0.06cm);
\draw[color=black, left] (0,8.85) node {$t_{3}$};
\draw[color=black, below] (7.85,0) node {$t_{2}$};
\draw[color=black, left] (4.15,-4.15) node {$t_{1}$};
\draw[color=black, left] (1.75,-1.75) node {$a$};
\draw[color=black, above] (3.2,0) node {$b$};
\draw[color=black, left] (0,4) node {$c$};
\draw[color=black, above] (4.75,2.25) node {$(a,b,c)$};
\end{tikzpicture}%
%EndExpansion
%

%TCIMACRO{\TeXButton{Fig4}{\begin{tikzpicture}[x=0.3cm,y=0.3cm]
%\draw
%[line width=0.5pt,color=blue!50,fill=blue!50] (4.75,2.25)--(4.75,6.25)--(8.75,6.25)--
%(10.5,4.5)--(10.5,0.5)--(6.5,0.5)--(4.75,2.25);
%\draw
%(4.75,2.25)--(4.75,6.25)--(8.75,6.25)--(10.5,4.5)--(10.5,0.5)--(6.5,0.5)--(4.75,2.25);
%\draw(4.75,6.25)--(6.5,4.5)--(10.5,4.5);
%\draw(6.5,0.5)--(6.5,4.5);
%\draw[dashed] (4.75,2.25)--(8.75,2.25)--(10.5,0.5);
%\draw[dashed] (8.75,2.25)--(8.75,6.25);
%\draw[dashed] (1.75,-1.75)--(8.75,-1.75);
%\draw[dashed] (3,0)--(6.5,-3.5);
%\draw[dashed] (3.5,-3.5)--(10.5,-3.5);
%\draw[dashed] (7,0)--(10.5,-3.5);
%\draw[dashed] (4.75,-1.75)--(4.75,2.25);
%\draw[dashed] (8.75,-1.75)--(8.75,2.25);
%\draw[dashed] (6.5,-3.5)--(6.5,0.5);
%\draw[dashed] (10.5,-3.5)--(10.5,0.5);
%\draw[->] (0,0)--(8,0);
%\draw[->] (0,0)--(0,9);
%\draw[->] (0,0)--(4.25,-4.25);
%\draw[line width=0.5pt,color=black,fill=black] (0,4) circle (0.06cm);
%\draw[line width=0.5pt,color=black,fill=black] (3,0) circle (0.06cm);
%\draw[line width=0.5pt,color=black,fill=black] (1.75,-1.75) circle (0.06cm);
%\draw[line width=0.5pt,color=blue,fill=blue] (4.75,2.25) circle (0.06cm);
%\draw[color=black, left] (0,8.85) node {$t_{3}$};
%\draw[color=black, below] (7.85,0) node {$t_{2}$};
%\draw[color=black, left] (4.15,-4.15) node {$t_{1}$};
%\draw[color=black, left] (1.75,-1.75) node {$a$};
%\draw[color=black, above] (3.2,0) node {$b$};
%\draw[color=black, left] (0,4) node {$c$};
%\draw[color=black, left] (4.75,2.25) node {$(a,b,c)$};
%\end{tikzpicture}}}%
%BeginExpansion
\begin{tikzpicture}[x=0.3cm,y=0.3cm]
\draw
[line width=0.5pt,color=blue!50,fill=blue!50] (4.75,2.25)--(4.75,6.25)--(8.75,6.25)--
(10.5,4.5)--(10.5,0.5)--(6.5,0.5)--(4.75,2.25);
\draw
(4.75,2.25)--(4.75,6.25)--(8.75,6.25)--(10.5,4.5)--(10.5,0.5)--(6.5,0.5)--(4.75,2.25);
\draw(4.75,6.25)--(6.5,4.5)--(10.5,4.5);
\draw(6.5,0.5)--(6.5,4.5);
\draw[dashed] (4.75,2.25)--(8.75,2.25)--(10.5,0.5);
\draw[dashed] (8.75,2.25)--(8.75,6.25);
\draw[dashed] (1.75,-1.75)--(8.75,-1.75);
\draw[dashed] (3,0)--(6.5,-3.5);
\draw[dashed] (3.5,-3.5)--(10.5,-3.5);
\draw[dashed] (7,0)--(10.5,-3.5);
\draw[dashed] (4.75,-1.75)--(4.75,2.25);
\draw[dashed] (8.75,-1.75)--(8.75,2.25);
\draw[dashed] (6.5,-3.5)--(6.5,0.5);
\draw[dashed] (10.5,-3.5)--(10.5,0.5);
\draw[->] (0,0)--(8,0);
\draw[->] (0,0)--(0,9);
\draw[->] (0,0)--(4.25,-4.25);
\draw[line width=0.5pt,color=black,fill=black] (0,4) circle (0.06cm);
\draw[line width=0.5pt,color=black,fill=black] (3,0) circle (0.06cm);
\draw[line width=0.5pt,color=black,fill=black] (1.75,-1.75) circle (0.06cm);
\draw[line width=0.5pt,color=blue,fill=blue] (4.75,2.25) circle (0.06cm);
\draw[color=black, left] (0,8.85) node {$t_{3}$};
\draw[color=black, below] (7.85,0) node {$t_{2}$};
\draw[color=black, left] (4.15,-4.15) node {$t_{1}$};
\draw[color=black, left] (1.75,-1.75) node {$a$};
\draw[color=black, above] (3.2,0) node {$b$};
\draw[color=black, left] (0,4) node {$c$};
\draw[color=black, left] (4.75,2.25) node {$(a,b,c)$};
\end{tikzpicture}%
%EndExpansion
%

%TCIMACRO{\TeXButton{%
%\end{multicols}%
%}{\end{multicols}}}%
%BeginExpansion
\end{multicols}%
%EndExpansion
%

%TCIMACRO{\TeXButton{minipage}{\begin{minipage}[bl]{9.05cm}
%\begin{enumerate}
%\item[{\color{red!50}$\blacksquare$}%
%] Non admissible exponents for the HL inequality
%\item[{\color{blue!50}$\blacksquare$}%
%] Admissible exponents for the HL inequality
%\item[{\color{gray!40}$\blacksquare$}] Unknown
%\end{enumerate}
%\end{minipage}}}%
%BeginExpansion
\begin{minipage}[bl]{9.05cm}
\begin{enumerate}
\item[{\color{red!50}$\blacksquare$}%
] Non admissible exponents for the HL inequality
\item[{\color{blue!50}$\blacksquare$}%
] Admissible exponents for the HL inequality
\item[{\color{gray!40}$\blacksquare$}] Unknown
\end{enumerate}
\end{minipage}%
%EndExpansion
%

%TCIMACRO{\TeXButton{%
%\end{figure}%
%}{\end{figure}}}%
%BeginExpansion
\end{figure}%
%EndExpansion

Despite the active research in the field, several basic issues remain open and
in the present paper we deal with some of these questions. Our first main
result (Theorem \ref{9876}) provides globally sharp exponents for the case
$1/2\leq1/p_{1}+\cdots+1/p_{m}<1$. We begin by recalling recent results in
this line.

The first one is due to Osikiewicz and Tonge (Theorem \ref{tonge}); in 2016,
Dimant and Sevilla-Peris (Theorem \ref{dddd}) obtained an optimal $m$-linear
version for the isotropic version (choosing the same exponents in all indexes)
and, more recently, anisotropic variants were obtained by Albuquerque and
Rezende (Theorem \ref{nnn}) and \ Aron \textit{et al}. (Theorem \ref{aaa}).

\begin{theorem}
$\left(  \text{See \cite[Theorem 5]{tonge}}\right)  $\label{tonge} If
$p_{1},p_{2}\in(2,\infty]\times(1,2]$ and $1/2\leq1/p_{1}+1/p_{2}<1$, then
\[
\left(  \sum_{j_{1}=1}^{n}\left(  \sum_{j_{2}=1}^{n}\left\vert A(e_{j_{1}%
},e_{j_{2}})\right\vert ^{p_{2}^{\ast}}\right)  ^{\frac{\lambda}{p_{2}^{\ast}%
}}\right)  ^{\frac{1}{\lambda}}\leq\left\Vert A\right\Vert \text{,}%
\]
for all bilinear forms $A\colon\ell_{p_{1}}^{n}\times\ell_{p_{2}}%
^{n}\rightarrow\mathbb{K}$, with $1/\lambda:=1-\left(  1/p_{1}+1/p_{2}\right)
$.
\end{theorem}

\begin{theorem}
$($See \cite[Proposition 4.1]{dimant}$)$\label{dddd} If $p_{1},\ldots,p_{m}%
\in\lbrack1,\infty]$ and $1/2\leq1/p_{1}+\cdots+1/p_{m}<1,$ then there is a
constant $C_{m}$ such that%
\[
\left(  \sum_{j_{1}=1}^{n}\cdots\sum_{j_{m}=1}^{n}\left\vert A\left(
e_{j_{1}},\ldots,e_{j_{m}}\right)  \right\vert ^{\lambda}\right)  ^{\frac
{1}{\lambda}}\leq C_{m}\left\Vert A\right\Vert \text{,}%
\]
for all $m$-linear forms $A\colon\ell_{p_{1}}^{n}\times\cdots\times\ell
_{p_{m}}^{n}\rightarrow\mathbb{K}$, with $1/\lambda:=1-\left(  1/p_{1}%
+\cdots+1/p_{m}\right)  $. Moreover, $\lambda$ is optimal.
\end{theorem}

\begin{theorem}
$($See \cite[Corollary 2]{NacLis}$)$\label{nnn} If $p_{1},\ldots,p_{m}%
\in(1,2m]$ and $1/2\leq1/p_{1}+\cdots+1/p_{m}<1$, then there is a constant
$C_{m}$ such that
\[
\left(  \sum_{j_{1}=1}^{n}\left(  \cdots\left(  \sum_{j_{m}=1}^{n}\left\vert
A\left(  e_{j_{1}},\ldots,e_{j_{m}}\right)  \right\vert ^{s_{m}}\right)
^{\frac{s_{m-1}}{s_{m}}}\cdots\right)  ^{\frac{s_{2}}{s_{1}}}\right)
^{\frac{1}{s_{1}}}\leq C_{m}\left\Vert A\right\Vert
\]
for all $m$-linear forms $A\colon\ell_{p_{1}}^{n}\times\cdots\times\ell
_{p_{n}}^{n}\rightarrow\mathbb{K}$, with%
\[
s_{k}=\left[  \dfrac{1}{2}+\frac{m-k+1}{2m}-\left(  \frac{1}{p_{k}}%
+\cdots+\frac{1}{p_{m}}\right)  \right]  ^{-1}\text{,}%
\]
for all $k=1,\ldots,m$.
\end{theorem}

\begin{theorem}
$($See \cite[Theorem 3.2]{AAPR}$)$\label{aaa} If $p_{1},\ldots,p_{m-1}%
\in(1,\infty]$, $p_{m}\in(1,2]$ and $1/2\leq1/p_{1}+\cdots+1/p_{m}<1$, then
\[
\left(  \sum_{j_{1}=1}^{n}\left(  \sum_{j_{2}=1}^{n}\cdots\left(  \sum
_{j_{m}=1}^{n}\left\vert A(e_{j_{1}},\ldots,e_{j_{m}})\right\vert ^{s_{m}%
}\right)  ^{\frac{s_{m-1}}{s_{m}}}\cdots\right)  ^{\frac{s_{1}}{s_{2}}%
}\right)  ^{\frac{1}{s_{1}}}\leq\left\Vert A\right\Vert \text{,}%
\]
for all $m$-linear forms $A\colon\ell_{p_{1}}^{n}\times\cdots\times\ell
_{p_{m}}^{n}\rightarrow\mathbb{K}$, where%
\[
s_{k}=\left[  1-\left(  \frac{1}{p_{k}}+\cdots+\frac{1}{p_{m}}\right)
\right]  ^{-1}\text{,}%
\]
for all $k=1,\ldots,m$. Moreover the exponents $s_{1},\ldots,s_{m}$ are optimal.
\end{theorem}

Our first main theorem, stated below, encompasses/generalizes/extends all the
previous results in an essentially optimal fashion. More precisely, we extend
the theorem of Osikiewicz and Tonge (Theorem \ref{tonge}) to the multilinear
setting, improve the exponents provided by Theorems \ref{dddd} and \ref{nnn},
and relax the hypothesis $1<p_{m}\leq2<p_{1},\ldots,p_{m-1}$ of Theorem
\ref{aaa}; note that we offer a different and simplified proof of Theorem
\ref{aaa}. We also recover, by a completely different approach, the optimal
constants from Theorem \ref{tonge} and Theorem \ref{aaa} (see Remark \ref{8z}).

\begin{theorem}
\label{9876}Let $p_{1},\ldots,p_{m}\in\left[  1,\infty\right]  $ be such that
$1/2\leq1/p_{1}+\cdots+1/p_{m}<1$. There is a constant $C_{m}$ such that
\begin{equation}
\left(  \sum_{j_{1}=1}^{n}\left(  \cdots\left(  \sum_{j_{m}=1}^{n}\left\vert
A\left(  e_{j_{1}},\ldots,e_{j_{m}}\right)  \right\vert ^{s_{m}}\right)
^{\frac{s_{m}-1}{s_{m}}}\cdots\right)  ^{\frac{s_{1}}{s_{2}}}\right)
^{\frac{1}{s_{1}}}\leq C_{m}\left\Vert A\right\Vert \text{,} \label{BER}%
\end{equation}
for all $m$-linear forms $A\colon\ell_{p_{1}}^{n}\times\cdots\times\ell
_{p_{m}}^{n}\rightarrow\mathbb{K}$, where%
\begin{equation}
s_{k}=\left\{
\begin{array}
[c]{ll}%
\left[  1-\left(  \dfrac{1}{p_{k}}+\cdots+\dfrac{1}{p_{m}}\right)  \right]
^{-1}\text{,} & \text{if }k\leq k_{0}:=\max\left\{  t:\dfrac{1}{p_{t}}%
+\cdots+\dfrac{1}{p_{m}}\geq\dfrac{1}{2}\right\}  \text{,}\\
\vspace{-0.3cm} & \\
2\text{,} & \text{if }k>k_{0}\text{.}%
\end{array}
\right.  \label{strelanatal}%
\end{equation}
Moreover:

\begin{enumerate}
\item[\emph{(i)}] The exponents $s_{1},\ldots,s_{k_{0}}$ are optimal;

\item[\emph{(ii)}] If $p_{k_{0}}\geq2$, the exponents $\left(  s_{1}%
,\ldots,s_{m}\right)  $ are globally sharp.
\end{enumerate}
\end{theorem}

Notice that, under the hypotheses of the previous theorem, we have $p_{k}>2$
for all $k=k_{0}+1,\ldots,m$. In fact, by the definition of $k_{0}$, we have
$1/p_{k_{0}+1}+\cdots+1/p_{m}<1/2$ and hence $p_{k}>2$ for all $k=k_{0}%
+1,\ldots,m$.

The improvement of Theorems \ref{tonge}, \ref{dddd} and \ref{nnn} is easily
observed from the statement of our main theorem (see also Table \ref{2dez}).
As to Theorem \ref{aaa}, note that an immediate corollary of Theorem
\ref{9876} yields the following result, that recovers Theorem \ref{aaa} for
the particular case $i=m$:

\begin{corollary}
Let $p_{1},\ldots,p_{m}\in\left[  1,\infty\right]  $ be such that
$1/2\leq1/p_{1}+\cdots+1/p_{m}<1$ with $1<p_{i}\leq2$ for a certain $i$. There
is a constant $C_{m}$ such that%
\[
\left(  \sum_{j_{1}=1}^{n}\left(  \sum_{j_{2}=1}^{n}\cdots\left(  \sum
_{j_{m}=1}^{n}\left\vert A\left(  e_{j_{1}},\ldots,e_{j_{m}}\right)
\right\vert ^{s_{m}}\right)  ^{\frac{s_{m-1}}{s_{m}}}\cdots\right)
^{\frac{s_{1}}{s_{2}}}\right)  ^{\frac{1}{s_{1}}}\leq C_{m}\left\Vert
A\right\Vert \text{,}%
\]
for all $m$-linear forms $A\colon\ell_{p_{1}}^{n}\times\cdots\times\ell
_{p_{m}}^{n}\rightarrow\mathbb{K}$, where%
\[
s_{k}=\left\{
\begin{array}
[c]{ll}%
\left[  1-\left(  \dfrac{1}{p_{k}}+\cdots+\dfrac{1}{p_{m}}\right)  \right]
^{-1}\text{,} & \text{if }k\leq i\text{,}\\
\vspace{-0.3cm} & \\
2\text{,} & \text{if }k>i\text{.}%
\end{array}
\right.
\]
Moreover, the exponents $s_{1},\ldots,s_{i}$ are optimal.
\end{corollary}

Our second main result is related to a vector-valued version of the HL
inequalities. As it will be explained later, it is a kind of extension of
\cite[Theorem 2.2]{alb2}. More precisely, it reads as follows (the precise
definitions of the notions and terminology used in its statement will be
defined in the next section):

\begin{theorem}
\label{vector}Let $p_{1},\ldots,p_{m}\in\left[  1,\infty\right]  $, $E$ be a
Banach space, $F$ be a Banach space of cotype $q$ and $1\leq r\leq q$, with
\[
\frac{1}{r}-\frac{1}{q}\leq\frac{1}{p_{1}}+\cdots+\frac{1}{p_{m}}<\frac{1}%
{r}\text{.}%
\]
Then, there is a constant $C_{m}$ such that
\[
\left(  \sum_{j_{1}=1}^{n}\left(  \cdots\left(  \sum_{j_{m}=1}^{n}\left\Vert
vA\left(  e_{j_{1}},\ldots,e_{j_{m}}\right)  \right\Vert _{F}^{s_{m}}\right)
^{\frac{s_{m}-1}{s_{m}}}\cdots\right)  ^{\frac{s_{1}}{s_{2}}}\right)
^{\frac{1}{s_{1}}}\leq C_{m}\left\Vert A\right\Vert \pi_{r,1}\left(  v\right)
\]
for all $m$-linear operators $A\colon\ell_{p_{1}}^{n}\times\cdots\times
\ell_{p_{m}}^{n}\rightarrow E$ and all absolutely $\left(  r,1\right)
$-summing operators $v\colon E\rightarrow F$, with%
\[
s_{k}=\left\{
\begin{array}
[c]{ll}%
\left[  \dfrac{1}{r}-\left(  \dfrac{1}{p_{k}}+\cdots+\dfrac{1}{p_{m}}\right)
\right]  ^{-1}\text{,} & \text{if }k\leq k_{0}:=\max\left\{  t:\dfrac{1}%
{p_{t}}+\cdots+\dfrac{1}{p_{m}}\geq\dfrac{1}{r}-\dfrac{1}{q}\right\}
\text{,}\\
\vspace{-0.3cm} & \\
q\text{,} & \text{if }k>k_{0}\text{.}%
\end{array}
\right.
\]

\end{theorem}

The paper is organized as follows. In Section $2$ we present the main tools
(the Anisotropic Regularity Principle and techniques of the theory of
multilinear absolutely summing operators) that will be used along the proofs
of our main results. In Sections $3$ and $4$ we prove our first main theorem.
In Section $5$ we prove the second main theorem and, in Section $6$, we obtain
globally sharp exponents for the critical case $p_{1}=\cdots=p_{m}=m$.
Finally, in the final section we show that, in general, the Anisotropic
Regularity Principle is optimal.

\section{Regularity Principle: the main tool}

Regularity arguments are present in different contexts of Mathematics, in the
search of optimal parameters in problems from PDEs to Classical Analysis.
Heuristically, regularity principles are usually hidden in subtle optimization
problems; see for instance\ \cite{1,2,3} for a select account in the realm of
diffusive PDEs.

A Regularity treatment of Hardy--Littlewood inequalities was successfully
launched in \cite{PSRT}, where the authors investigated the following general
universality problem (observe that the existence of a leeway, $\epsilon>0$, of
an increment $\delta>0$, and of a corresponding bound $\tilde{C}%
_{\delta,\epsilon}>0$ carries a regularity principle):

\begin{problem}
\label{Univ. Problem} Let $p\geq1$ be a real number, $X,Y,W_{1},W_{2}$ be
non-void sets, $Z_{1},Z_{2},Z_{3}$ be normed spaces and $f\colon X\times
Y\rightarrow Z_{1},\ g\colon X\times W_{1}\rightarrow Z_{2},\ h\colon Y\times
W_{2}\rightarrow Z_{3}$ be particular maps. Assume there is a constant $C>0$
such that%
\begin{equation}
{%
%TCIMACRO{\tsum \limits_{i=1}^{m_{1}}}%
%BeginExpansion
{\textstyle\sum\limits_{i=1}^{m_{1}}}
%EndExpansion%
%TCIMACRO{\tsum \limits_{j=1}^{m_{2}}}%
%BeginExpansion
{\textstyle\sum\limits_{j=1}^{m_{2}}}
%EndExpansion
\left\Vert f(x_{i},y_{j})\right\Vert ^{p}\leq C\left(  \sup_{w\in W_{1}}%
%TCIMACRO{\tsum \limits_{i=1}^{m_{1}}}%
%BeginExpansion
{\textstyle\sum\limits_{i=1}^{m_{1}}}
%EndExpansion
\left\Vert g(x_{i},w)\right\Vert ^{p}\right)  \cdot\left(  \sup_{w\in W_{2}}%
%TCIMACRO{\tsum \limits_{j=1}^{m_{2}}}%
%BeginExpansion
{\textstyle\sum\limits_{j=1}^{m_{2}}}
%EndExpansion
\left\Vert h(y_{j},w)\right\Vert ^{p}\right)  }, \label{ree}%
\end{equation}
for all $x_{i}\in X$, $y_{j}\in Y$ and $m_{1},m_{2}\in\mathbb{N}$. Are there
(universal) positive constants $\epsilon\sim\delta$, and $\tilde{C}%
_{\delta,\epsilon}$ such that%
\begin{equation}
\text{ }\left(  {%
%TCIMACRO{\tsum \limits_{i=1}^{m_{1}}}%
%BeginExpansion
{\textstyle\sum\limits_{i=1}^{m_{1}}}
%EndExpansion%
%TCIMACRO{\tsum \limits_{j=1}^{m_{2}}}%
%BeginExpansion
{\textstyle\sum\limits_{j=1}^{m_{2}}}
%EndExpansion
}\left\Vert f(x_{i},y_{j})\right\Vert ^{p+\delta}\right)  ^{\frac{1}{p+\delta
}}\leq\tilde{C}_{\delta,\epsilon}\cdot\left(  \sup_{w\in W_{1}}%
%TCIMACRO{\tsum \limits_{i=1}^{m_{1}}}%
%BeginExpansion
{\textstyle\sum\limits_{i=1}^{m_{1}}}
%EndExpansion
\left\Vert g(x_{i},w)\right\Vert ^{p+\epsilon}\right)  ^{\frac{1}{p+\epsilon}%
}\left(  \sup_{w\in W_{2}}%
%TCIMACRO{\tsum \limits_{j=1}^{m_{2}}}%
%BeginExpansion
{\textstyle\sum\limits_{j=1}^{m_{2}}}
%EndExpansion
\left\Vert h(y_{j},w)\right\Vert ^{p+\epsilon}\right)  ^{\frac{1}{p+\epsilon}%
}, \label{ree1}%
\end{equation}
for all $x_{i}\in X$, $y_{j}\in Y$ and $m_{1},m_{2}\in\mathbb{N}$?
\end{problem}

The answer to Problem \ref{Univ. Problem} presented in \cite{PSRT} was
obtained for a wide class of functions (note that continuity is not needed).
We just need few mild assumptions. Let $\,Z_{1},\,V$ and
$W_{1},\,W_{2}$ be non-void sets and $Z_{2}$ be a vector space. For $t=1,2,$
let%
\[
R_{t}\colon Z_{t}\times W_{t}\longrightarrow\lbrack0,\infty)\text{ and
}S\colon Z_{1}\times Z_{2}\times V\longrightarrow\lbrack0,\infty)
\]
\ be mappings satisfying%
\[
\left\{
\begin{array}
[c]{l}%
R_{2}\left(  \beta z,w\right)  =\beta R_{2}\left(  z,w\right)  ,\\
S\left(  z_{1},\beta z_{2},v\right)  =\beta S\left(  z_{1},z_{2},v\right)
\end{array}
\right.
\]

for all real scalars $\beta\geq0.$

\begin{theorem}
[Regularity Principle \cite{PSRT}]\label{RP} Let $1\leq p_{1}\leq p_{2}%
:=p_{1}+\epsilon<2p_{1}$ and assume
\[
\left(  \sup_{v\in V}\sum_{i=1}^{m_{1}}\sum_{j=1}^{m_{2}}S(z_{1,i}%
,z_{2,j},v)^{p_{1}}\right)  ^{\frac{1}{p_{1}}}\hspace{-0.2cm}\leq C\left(
\sup_{w\in W_{1}}\sum_{i=1}^{m_{1}}R_{1}\left(  z_{1,i},w\right)  ^{p_{1}%
}\right)  ^{\frac{1}{p_{1}}}\left(  \sup_{w\in W_{2}}\sum_{j=1}^{m_{2}}%
R_{2}\left(  z_{2,j},w\right)  ^{p_{1}}\right)  ^{\frac{1}{p_{1}}},
\]
for all $z_{1,i}\in Z_{1},z_{2,j}\in Z_{2},$ all $i=1,...,m_{1}$ and
$j=1,...,m_{2}$ and $m_{1},m_{2}\in\mathbb{N}$. Then%
\[
\left(  \sup_{v\in V}\sum_{i=1}^{m_{1}}\sum_{j=1}^{m_{2}}S(z_{1,i}%
,z_{2,j},v)^{\alpha}\right)  ^{\frac{1}{\alpha}}\hspace{-0.3cm}\leq C\left(
\sup_{w\in W_{1}}\sum_{i=1}^{m_{1}}R_{1}\left(  z_{1,i},w\right)  ^{p_{2}%
}\right)  ^{\frac{1}{p_{2}}}\left(  \sup_{w\in W_{2}}\sum_{j=1}^{m_{2}}%
R_{2}\left(  z_{2,j},w\right)  ^{p_{2}}\right)  ^{\frac{1}{p_{2}}},
\]
for
\[
\alpha=\frac{p_{1}p_{2}}{2p_{1}-p_{2}},
\]
all $z_{1,i}\in Z_{1},z_{2,j}\in Z_{2},$ all $i=1,...,m_{1}$ and
$j=1,...,m_{2}$ and $m_{1},m_{2}\in\mathbb{N}$.
\end{theorem}

\begin{remark}
Note that the above theorem shows that, in Problem \ref{Univ. Problem}, for
$\epsilon<p_{1}$, we can choose
\begin{equation}
\delta=\frac{2\epsilon p_{1}}{p_{1}-\epsilon}. \label{deltadez}%
\end{equation}

\end{remark}

In this section we present some preliminary notions that will be used along
the paper, together with the regularity techniques originated from \cite{PSRT}.

Let $E_{1},\ldots,E_{m},F$ be Banach spaces and $\mathcal{L}\left(
E_{1},\ldots,E_{m};F\right)  $ denote the space of all continuous $m$-linear
operators $A\colon E_{1}\times\cdots\times E_{m}\rightarrow F$. If $\left(
r_{1},\ldots,r_{m}\right)  ,\left(  p_{1},\ldots,p_{m}\right)  \in\left[
1,\infty\right)  ^{m}$, an $m$-linear operator $A\in\mathcal{L}\left(
E_{1},\ldots,E_{m};F\right)  $ is multiple $\left(  r_{1},\ldots,r_{m}%
;p_{1},\ldots,p_{m}\right)  $-summing if there is a constant $C_{m}$ such
that
\[
\left(  \sum_{j_{1}=1}^{n}\left(  \cdots\left(  \sum_{j_{m}=1}^{n}\left\Vert
A\left(  x_{j_{1}}^{\left(  1\right)  },\ldots,x_{j_{m}}^{\left(  m\right)
}\right)  \right\Vert ^{r_{m}}\right)  ^{\frac{r_{m-1}}{r_{m}}}\cdots\right)
^{\frac{r_{1}}{r_{2}}}\right)  ^{\frac{1}{r_{1}}}\leq C_{m}\prod_{k=1}^{m}%
\sup_{\varphi_{k}\in B_{E_{k}^{\ast}}}\left(  \sum_{j=1}^{n}\left\vert
\varphi_{k}\left(  x_{j}^{\left(  k\right)  }\right)  \right\vert ^{p_{k}%
}\right)  ^{\frac{1}{p_{k}}}%
\]
for all positive integers $n$ (above, $E_{k}^{\ast}$ represents the
topological dual of $E_{k}$ and $B_{E_{k}^{\ast}}$ represents its closed unit ball).

The class of all multiple $\left(  r_{1},\ldots,r_{m};p_{1},\ldots
,p_{m}\right)  $-summing operators $A\colon E_{1}\times\cdots\times
E_{m}\rightarrow F$ is denoted by $\Pi_{\left(  r_{1},\ldots,r_{m}%
;p_{1},\ldots,p_{m}\right)  }^{m}\left(  E_{1},\ldots,E_{m};F\right)  $; when
$m=1$ we write $\Pi_{\left(  r_{1};p_{1}\right)  }$ instead of $\Pi_{\left(
r_{1};p_{1}\right)  }^{m}$. When $r_{1}=\cdots=r_{m}=r$ and $p_{1}%
=\cdots=p_{m}=p$, we simply write $\left(  r;p_{1},\ldots,p_{m}\right)  $ and
$\left(  r_{1},\ldots,r_{m};p\right)  $, respectively. The infimum of the
constants $C_{m}$ defines a complete norm for the space $\Pi_{\left(
r_{1},\ldots,r_{m};p_{1},\ldots,p_{m}\right)  }^{m}\left(  E_{1},\ldots
,E_{m};F\right)  $, denoted hereafter by $\pi_{(r_{1},\ldots,r_{m};p_{1}%
,\dots,p_{m})}(\cdot)$. In the case of linear operators, the notion of
multiple summing operators reduces to the well-known concept of absolutely
summing operators, which plays a fundamental role in Banach Space Theory (see
the excellent monograph \cite{diestel}).

For recent results on absolutely summing multilinear operators we refer to the
papers \cite{banakh,b,bfr,PSRT} and the references therein. The following well-known
result, that will be used several times in this paper, associates HL
inequalities and multiple summing operators (see, for instance, \cite[Theorem
3.2]{PSRT} and the references therein). It essentially asserts that each HL
inequality corresponds to a coincidence result for multiple summing operators
which holds regardless of the Banach spaces considered at the domain:

\medskip%

%TCIMACRO{\TeXButton{%
%\begin{framed}%
%}{\begin{framed}}}%
%BeginExpansion
\begin{framed}%
%EndExpansion

\noindent\textbf{Hardy--Littlewood inequalities vs multiple summing
operators.} If $\left(  p_{1},\ldots,p_{m}\right)  \in\left[  1,\infty\right]
^{m}$ and $C\geq1$, the following statements are equivalent:

\begin{enumerate}
\item[(i)]
\begin{equation}
\left(  \sum_{j_{1}=1}^{n}\left(  \cdots\left(  \sum_{j_{m}=1}^{n}\left\Vert
A\left(  e_{j_{1}},\ldots,e_{j_{m}}\right)  \right\Vert ^{t_{m}}\right)
^{\frac{t_{m}-1}{t_{m}}}\cdots\right)  ^{\frac{t_{1}}{t_{2}}}\right)
^{\frac{1}{t_{1}}}\leq C\left\Vert A\right\Vert \text{,} \label{trt}%
\end{equation}
for every $A\in\mathcal{L}\left(  \ell_{p_{1}}^{n},\ldots,\ell_{p_{m}}%
^{n};F\right)  $ and all positive integers $n$.

\bigskip

\item[(ii)] For all Banach spaces $E_{1},\ldots,E_{m},F$ we have
\begin{equation}
\mathcal{L}\left(  E_{1},\ldots,E_{m};F\right)  =\Pi_{\left(  t_{1}%
,\ldots,t_{m};p_{1}^{\ast},\ldots,p_{m}^{\ast}\right)  }^{m}\left(
E_{1},\ldots,E_{m};F\right)  \label{trt2}%
\end{equation}
and%
\[
\pi_{(t_{1},\ldots,t_{m};p_{1}^{\ast},\dots,p_{m}^{\ast})}(\cdot)\leq
C\left\Vert \cdot\right\Vert \text{.}%
\]

\end{enumerate}

%

%TCIMACRO{\TeXButton{%
%\end{framed}%
%}{\end{framed}}}%
%BeginExpansion
\end{framed}%
%EndExpansion

\medskip

The next result also plays a fundamental role in this paper. It was
ingeniously obtained by Albuquerque and Rezende \cite[Theorem 3]{NacLis} as a
consequence of the Regularity Principle\ (see also \cite[Theorem 4]{NacLis}
for an anisotropic variant of Theorem \ref{RP}):

\medskip%

%TCIMACRO{\TeXButton{%
%\begin{framed}%
%}{\begin{framed}}}%
%BeginExpansion
\begin{framed}%
%EndExpansion

\noindent\textbf{Anisotropic Regularity Principle for summing operators.} If
$r\geq1$, and%
\[
\left(  s_{1},\ldots,s_{m}\right)  ,\left(  p_{1},\ldots,p_{m}\right)
,\left(  q_{1},\ldots,q_{m}\right)  \in\left[  1,\infty\right)  ^{m}%
\]
are such that $q_{k}\geq p_{k}$, and%
\[
\left[
\begin{array}
[c]{l}%
\dfrac{1}{r}-\left(  \dfrac{1}{p_{1}}+\cdots+\dfrac{1}{p_{m}}\right)  +\left(
\dfrac{1}{q_{1}}+\cdots+\dfrac{1}{q_{m}}\right)  >0\text{,}\\
\vspace{-0.3cm}\\
\dfrac{1}{s_{k}}-\left(  \dfrac{1}{q_{k}}+\cdots+\dfrac{1}{q_{m}}\right)
=\dfrac{1}{r}-\left(  \dfrac{1}{p_{k}}+\cdots+\dfrac{1}{p_{m}}\right)
\text{,}%
\end{array}
\right.
\]
for all $k\in\left\{  1,\ldots,m\right\}  $, then
\begin{equation}
\Pi_{\left(  r;p_{1},\ldots,p_{m}\right)  }^{m}\left(  E_{1},\ldots
,E_{m};F\right)  \subset\Pi_{\left(  s_{1},\ldots,s_{m};q_{1},\ldots
,q_{m}\right)  }^{m}\left(  E_{1},\ldots,E_{m};F\right)  \text{,} \label{9080}%
\end{equation}
for any Banach spaces $E_{1},\ldots,E_{m}$ and $F$. Moreover, the inclusion
operator has norm $1$.%

%TCIMACRO{\TeXButton{%
%\end{framed}%
%}{\end{framed}}}%
%BeginExpansion
\end{framed}%
%EndExpansion

\medskip

\section{Proof of Theorem \ref{9876}: existence}

Given a positive integer $n$, let $A_{0}\colon\ell_{p_{1}}^{n}\times
\cdots\times\ell_{p_{m}}^{n}\rightarrow\mathbb{K}$ be an $m$-linear form and
$k_{0}$ be as in the statement of the theorem.

If $k_{0}=m$, then $1<p_{m}\leq2$. Denoting
\[
\left(  t_{1},\ldots,t_{m-1},t_{m}\right)  =\left(  \infty,\ldots,\infty
,p_{m}\right)  \text{,}%
\]
we obtain
\[
\dfrac{1}{t_{1}}+\cdots+\dfrac{1}{t_{m}}=\frac{1}{p_{m}}%
\]
and hence%
\[
\frac{1}{2}\leq\dfrac{1}{t_{1}}+\cdots+\dfrac{1}{t_{m}}<1\text{.}%
\]
Thus, using Theorem \ref{dddd}, there is $C_{m}$ such that%
\begin{equation}
\left(  \sum_{j_{1}=1}^{n}\cdots\sum_{j_{m}=1}^{n}\left\vert A\left(
e_{j_{1}},\ldots,e_{j_{m}}\right)  \right\vert ^{r_{m}}\right)  ^{r_{m}}\leq
C_{m}\left\Vert A\right\Vert \label{uno}%
\end{equation}
for all $m$-linear forms $A\colon\ell_{t_{1}}^{n}\times\cdots\times\ell
_{t_{m}}^{n}\rightarrow\mathbb{K}$, where $r_{m}=p_{m}^{\ast}$. Using the
equivalence $(\ref{trt})\Leftrightarrow\left(  \ref{trt2}\right)  $, this
means that every continuous $m$-linear form $A\colon E_{1}\times\cdots\times
E_{m}\rightarrow\mathbb{K}$ is multiple $\left(  r_{m};t_{1}^{\ast}%
,\ldots,t_{m}^{\ast}\right)  $-summing, regardless of the Banach spaces
$E_{1},\ldots,E_{m}$, and the associated constant is $C_{m}$. In particular,
$A_{0}\colon\ell_{p_{1}}^{n}\times\cdots\times\ell_{p_{m}}^{n}\rightarrow
\mathbb{K}$ is multiple $\left(  r_{m};t_{1}^{\ast},\ldots,t_{m}^{\ast
}\right)  $-summing, and the associated constant is $C_{m}$. A straightforward
calculation shows that%
\[
\frac{1}{r_{m}}-\left(  \dfrac{1}{t_{1}^{\ast}}+\cdots+\dfrac{1}{t_{m}^{\ast}%
}\right)  +\left(  \dfrac{1}{p_{1}^{\ast}}+\cdots+\dfrac{1}{p_{m}^{\ast}%
}\right)  >0\text{.}%
\]
For each $k=1,\ldots,m$, let%
\[
r_{k}=\left[  1-\left(  \dfrac{1}{p_{k}}+\cdots+\dfrac{1}{p_{m}}\right)
\right]  ^{-1}\text{.}%
\]
By the Anisotropic Regularity Principle we conclude that $A_{0}$ is multiple
$\left(  r_{1},\ldots,r_{m};p_{1}^{\ast},\ldots,p_{m}^{\ast}\right)  $-summing
with the same constant $C_{m}$, and this is equivalent to
\[
\left(  \sum_{j_{1}=1}^{n}\left(  \cdots\left(  \sum_{j_{m}=1}^{n}\left\vert
A_{0}\left(  e_{j_{1}},\ldots,e_{j_{m}}\right)  \right\vert ^{r_{m}}\right)
^{\frac{r_{m}-1}{r_{m}}}\cdots\right)  ^{\frac{r_{1}}{r_{2}}}\right)
^{\frac{1}{r_{1}}}\leq C_{m}\left\Vert A_{0}\right\Vert \text{.}%
\]
This proves (\ref{BER}), for the case $k_{0}=m$.

So, let us consider $k_{0}<m$. It is obvious that $p_{k_{0}}<\infty$ and
\[
\frac{1}{p_{k_{0}+1}}+\cdots+\frac{1}{p_{m}}<\frac{1}{2}\text{.}%
\]
Let $\delta\geq1$ be such that%
\[
\dfrac{1}{\delta p_{k_{0}}}+\dfrac{1}{p_{k_{0}+1}}+\cdots+\dfrac{1}{p_{m}%
}=\frac{1}{2}\text{.}%
\]
Denoting
\[
\left\{
\begin{array}
[c]{l}%
\left(  q_{1},\ldots,q_{k_{0}-1}\right)  =\left(  \infty,\ldots,\infty\right)
\text{,}\\
\vspace{-0.3cm}\\
\left(  q_{k_{0}},q_{k_{0}+1},\ldots,q_{m}\right)  =\left(  \delta p_{k_{0}%
},p_{k_{0}+1},\ldots,p_{m}\right)  \text{,}%
\end{array}
\right.
\]
we have
\[
\dfrac{1}{q_{1}}+\cdots+\dfrac{1}{q_{m}}=\frac{1}{2}%
\]
and, hence, by (\ref{i99}), there is a constant $C_{m}$ such that%
\begin{equation}
\left(  \sum_{j_{1}=1}^{n}\cdots\sum_{j_{m}=1}^{n}\left\vert A\left(
e_{j_{1}},\ldots,e_{j_{m}}\right)  \right\vert ^{2}\right)  ^{\frac{1}{2}}\leq
C_{m}\left\Vert A\right\Vert \label{dos}%
\end{equation}
for all $m$-linear forms $A\colon\ell_{q_{1}}^{n}\times\cdots\times\ell
_{q_{m}}^{n}\rightarrow\mathbb{K}$. Again, using the canonical association
between the HL inequalities and multiple summing operators, we have that
$A_{0}\colon\ell_{p_{1}}^{n}\times\cdots\times\ell_{p_{m}}^{n}\rightarrow
\mathbb{K}$ is multiple $\left(  2;q_{1}^{\ast},\ldots,q_{m}^{\ast}\right)
$-summing, and the associated constant is $C_{m}$. Since%
\[
\dfrac{1}{2}-\left(  \dfrac{1}{q_{1}^{\ast}}+\cdots+\dfrac{1}{q_{m}^{\ast}%
}\right)  +\left(  \dfrac{1}{p_{1}^{\ast}}+\cdots+\dfrac{1}{p_{m}^{\ast}%
}\right)  >0\text{,}%
\]
considering%
\[
s_{k}=\left\{
\begin{array}
[c]{ll}%
\left[  1-\left(  \dfrac{1}{p_{k}}+\cdots+\dfrac{1}{p_{m}}\right)  \right]
^{-1}\text{,} & \text{if }k\leq k_{0}\text{,}\\
\vspace{-0.3cm} & \\
2\text{,} & \text{if }k>k_{0}\text{,}%
\end{array}
\right.
\]
by (\ref{9080}) we conclude that $A_{0}$ is multiple $\left(  s_{1}%
,\ldots,s_{m};p_{1}^{\ast},\ldots,p_{m}^{\ast}\right)  $-summing with the same
constant $C_{m}$, and again, this is equivalent to
\[
\left(  \sum_{j_{1}=1}^{n}\left(  \cdots\left(  \sum_{j_{m}=1}^{n}\left\vert
A_{0}\left(  e_{j_{1}},\ldots,e_{j_{m}}\right)  \right\vert ^{s_{m}}\right)
^{\frac{s_{m}-1}{s_{m}}}\cdots\right)  ^{\frac{s_{1}}{s_{2}}}\right)
^{\frac{1}{s_{1}}}\leq C_{m}\left\Vert A_{0}\right\Vert \text{.}%
\]
This proves (\ref{BER}) for the case $k_{0}<m$.

In order to illustrate (numerically) how Theorem \ref{9876} improves the
previous ones, let us consider $m=9$ with
\[
p_{1}=\cdots=p_{9}=10\text{.}%
\]
The table below compares the exponents provided our first main result (Theorem
\ref{9876}) and those from Theorems \ref{dddd} (\cite[Proposition 4.1]%
{dimant}), and Theorem \ref{nnn} (\cite[Corollary 2]{NacLis}):%

%TCIMACRO{\TeXButton{%
%\begin{table}%
%[H]}{\begin{table}[H]
%}}%
%BeginExpansion
\begin{table}[H]
%EndExpansion
%

%TCIMACRO{\TeXButton{\centering}{\centering}}%
%BeginExpansion
\centering
%EndExpansion
%

%TCIMACRO{\TeXButton{\caption}{\caption{}}}%
%BeginExpansion
\caption{}%
%EndExpansion

\label{2dez}%

%TCIMACRO{\TeXButton{Tabela1}{$\begin{tabular}
%[c]{|c|c|c|c|c|c|c|c|c|c|}\hline& $s_{1}$ & $s_{2}$ & $s_{3}$ & $s_{4}
%$ & $s_{5}$ & $s_{6}$ & $s_{7}$ &
%$s_{8}$ & $s_{9}$\\\hline\multicolumn{1}{|c|}{\cite[Proposition 4.1]{dimant}}
%& $10$ & $10$ & $10$ &
%$10$ & $10$ & $10$ & $10$ & $10$ & $10$\\\hline\multicolumn{1}{|c|}
%{\cite[Corollary 2]{NacLis}} & $10$ & $\simeq6.92$ &
%$\simeq5.29$ & $\simeq4.28$ & $3.6$ & $\simeq3.10$ & $\simeq2.72$ &
%$\simeq2.43$ & $\simeq2.19$\\\hline\multicolumn{1}{|c|}{Theorem \ref{9876}}
%& \color{red}$10$ & \color{red}$5$ & \color{red}$\simeq3.33$ & \color
%{red}$2.5$ &
%\color{red}$2$ & \color{blue}$2$ & \color{blue}$2$ & \color{blue}%
%$2$ & \color{blue}$2$\\\hline\end{tabular}
%\ $\\
%\begin{flushleft}
%{\color{red} $\blacksquare$} The exponents are sharp\\
%{\color{blue} $\blacksquare$}
%The exponents (combined with the exponents in red) are globally sharp
%\end{flushleft}
%}}%
%BeginExpansion
$\begin{tabular}
[c]{|c|c|c|c|c|c|c|c|c|c|}\hline& $s_{1}$ & $s_{2}$ & $s_{3}$ & $s_{4}
$ & $s_{5}$ & $s_{6}$ & $s_{7}$ &
$s_{8}$ & $s_{9}$\\\hline\multicolumn{1}{|c|}{\cite[Proposition 4.1]{dimant}}
& $10$ & $10$ & $10$ &
$10$ & $10$ & $10$ & $10$ & $10$ & $10$\\\hline\multicolumn{1}{|c|}
{\cite[Corollary 2]{NacLis}} & $10$ & $\simeq6.92$ &
$\simeq5.29$ & $\simeq4.28$ & $3.6$ & $\simeq3.10$ & $\simeq2.72$ &
$\simeq2.43$ & $\simeq2.19$\\\hline\multicolumn{1}{|c|}{Theorem \ref{9876}}
& \color{red}$10$ & \color{red}$5$ & \color{red}$\simeq3.33$ & \color
{red}$2.5$ &
\color{red}$2$ & \color{blue}$2$ & \color{blue}$2$ & \color{blue}%
$2$ & \color{blue}$2$\\\hline\end{tabular}
\ $\\
\begin{flushleft}
{\color{red} $\blacksquare$} The exponents are sharp\\
{\color{blue} $\blacksquare$}
The exponents (combined with the exponents in red) are globally sharp
\end{flushleft}
%EndExpansion
%

%TCIMACRO{\TeXButton{%
%\end{table}%
%}{\end{table}}}%
%BeginExpansion
\end{table}%
%EndExpansion

\begin{remark}
\label{8z}By \cite[Theorem 3.3]{pascal} we notice that $C_{m}$ in
$(\ref{uno})$ and $(\ref{dos})$ satisfies%
\[
C_{m}\leq2^{\frac{m-k_{0}}{2}}\text{.}%
\]
Thus, by our procedure, \emph{Theorem \ref{9876}} stands with constant
$C_{m}\leq2^{\frac{m-k_{0}}{2}}$ and, choosing $k_{0}=m$, we recover the
optimal constants from \emph{Theorem \ref{tonge} and Theorem \ref{aaa}}.
\end{remark}

\section{Proof of Theorem \ref{9876}: optimality}

In this section we shall prove (i) and (ii) of Theorem \ref{9876}. The
optimality of the exponents in (i) is a straightforward consequence of
\cite[Lemma 3.1]{AAPR}. The proof of (ii) depends on a probabilistic result
that nowadays is usually called Kahane--Salem--Zygmund inequality. It has its
origins in the 1970's with independent papers of different authors
(\cite{be2,be1,kah,man,var}; see also \cite{mastylo} for a modern approach).
We shall need the following variant of the Kahane--Salem--Zygmund that can be
found at \cite[Lemma 6.1]{alb}:

\begin{lemma}
[Kahane--Salem--Zygmund inequality]Let $p_{1},\ldots,p_{m}\in\left[
2,\infty\right]  $. Then there exists a constant $K_{m}$ such that, for all
positive integers $n$, there is an $m$-linear form $A\colon\ell_{p_{1}}%
^{n}\times\cdots\times\ell_{p_{m}}^{n}\rightarrow\mathbb{K}$ of the form
\[
A\left(  z^{\left(  1\right)  },\ldots,z^{\left(  m\right)  }\right)
=\sum_{j_{1}=1}^{n}\cdots\sum_{j_{m}=1}^{n}\pm z_{j_{1}}^{\left(  1\right)
}\cdots z_{j_{m}}^{\left(  m\right)  }\text{,}%
\]
with
\[
\Vert A\Vert\leq K_{m}n^{\frac{1}{2}+\left(  \frac{1}{2}-\frac{1}{p_{1}%
}\right)  +\cdots+\left(  \frac{1}{2}-\frac{1}{p_{m}}\right)  }\text{.}%
\]

\end{lemma}

If $p_{1},\dots,p_{m}\in\lbrack2,\infty]$ and $t_{1},\dots,t_{m}\in\left[
1,\infty\right)  $, a straightforward consequence of the above inequality is
that if there is a constant $C_{m}$ such that
\[
\left(  \sum_{i_{1}=1}^{n}\left(  \cdots\left(  \sum_{i_{m}=1}^{n}\left\vert
A\left(  e_{i_{1}},\dots,e_{i_{m}}\right)  \right\vert ^{t_{m}}\right)
^{\frac{t_{m-1}}{t_{m}}}\cdots\right)  ^{\frac{t_{1}}{t_{2}}}\right)
^{\frac{1}{t_{1}}}\leq C_{m}\left\Vert A\right\Vert
\]
for all $m$-linear forms $A\colon\ell_{p_{1}}^{n}\times\cdots\times\ell
_{p_{m}}^{n}\rightarrow\mathbb{K}$, then
\begin{equation}
\frac{1}{t_{1}}+\cdots+\frac{1}{t_{m}}\leq\frac{m+1}{2}-\left(  \frac{1}%
{p_{1}}+\cdots+\frac{1}{p_{m}}\right)  \text{.} \label{colombia}%
\end{equation}
We shall also use the following simple lemma, which seems to be folklore:

\begin{lemma}
\label{8u8u}Let $k\in\left\{  1,\ldots,m-1\right\}  $ and $\left(
p_{1},\ldots,p_{m}\right)  \in\left[  1,\infty\right]  ^{m}$. If there exists
a constant $C_{m}$ such that
\[
\left(  \sum_{j_{1}=1}^{n}\left(  \cdots\left(  \sum_{j_{m}=1}^{n}\left\vert
A_{1}\left(  e_{j_{1}},\ldots,e_{j_{m}}\right)  \right\vert ^{t_{m}}\right)
^{\frac{t_{m-1}}{t_{m}}}\cdots\right)  ^{\frac{t_{1}}{t_{2}}}\right)
^{\frac{1}{t_{1}}}\leq C_{m}\left\Vert A_{1}\right\Vert
\]
for all $m$-linear forms $A_{1}\colon\ell_{p_{1}}^{n}\times\cdots\times
\ell_{p_{m}}^{n}\rightarrow\mathbb{K}$, then%
\[
\left(  \sum_{j_{k+1}=1}^{n}\left(  \cdots\left(  \sum_{j_{m}=1}^{n}\left\vert
A_{2}\left(  e_{j_{k+1}},\ldots,e_{j_{m}}\right)  \right\vert ^{t_{m}}\right)
^{\frac{t_{m-1}}{t_{m}}}\cdots\right)  ^{\frac{t_{k+1}}{t_{k+2}}}\right)
^{\frac{1}{t_{k+1}}}\leq C_{m}\left\Vert A_{2}\right\Vert
\]
for all $\left(  m-k\right)  $-linear forms $A_{2}\colon\ell_{p_{k+1}}%
^{n}\times\cdots\times\ell_{p_{m}}^{n}\rightarrow\mathbb{K}$.
\end{lemma}

By (i) we know that the exponents $s_{1},\ldots,s_{k_{0}}$ in
(\ref{strelanatal}) are optimal, so it is obvious that they cannot be
perturbed to smaller exponents. To conclude the proof of (ii) it remains to
consider the exponents $s_{k_{0}+1},\ldots,s_{m}$.

Let us suppose that for a certain $i=k_{0}+1,\ldots,m$ there is $\varepsilon
_{i}>0$ and there is a constant $C_{m}$ such that%
\[
\left(  \sum_{j_{1}=1}^{n}\left(  \cdots\left(  \sum_{j_{m}=1}^{n}\left\vert
A\left(  e_{j_{1}},\ldots,e_{j_{m}}\right)  \right\vert ^{t_{m}}\right)
^{\frac{t_{m}-1}{t_{m}}}\cdots\right)  ^{\frac{t_{1}}{t_{2}}}\right)
^{\frac{1}{t_{1}}}\leq C_{m}\left\Vert A\right\Vert \text{,}%
\]
for all $m$-linear forms $A\colon\ell_{p_{1}}^{n}\times\cdots\times\ell
_{p_{m}}^{n}\rightarrow\mathbb{K}$, with
\[
t_{k}=\left\{
\begin{array}
[c]{ll}%
\left[  1-\left(  \dfrac{1}{p_{k}}+\cdots+\dfrac{1}{p_{m}}\right)  \right]
^{-1}\text{,} & \text{if }k\leq k_{0}\text{,}\\
\vspace{-0.3cm} & \\
2\text{,} & \text{if }k>k_{0}\text{ and }k\not =i\\
\vspace{-0.3cm} & \\
2-\varepsilon_{i}\text{,} & \text{if }k=i\text{.}%
\end{array}
\right.
\]
We invoke Lemma \ref{8u8u} to conclude that
\[
\left(  \sum_{j_{k_{0}}=1}^{n}\left(  \cdots\left(  \sum_{j_{m}=1}%
^{n}\left\vert A\left(  e_{j_{1}},\ldots,e_{j_{m}}\right)  \right\vert
^{t_{m}}\right)  ^{\frac{t_{m}-1}{t_{m}}}\cdots\right)  ^{\frac{t_{k_{0}}%
}{t_{k_{0}+1}}}\right)  ^{\frac{1}{t_{k_{0}}}}\leq C_{m}\left\Vert
A\right\Vert
\]
for all $\left(  m-k_{0}+1\right)  $-linear forms $A\colon\ell_{p_{k_{0}}}%
^{n}\times\cdots\times\ell_{p_{m}}^{n}\rightarrow\mathbb{K}$. Note that
\begin{align*}
\frac{1}{t_{k_{0}}}+\cdots+\frac{1}{t_{m}}  &  =1-\left(  \dfrac{1}{p_{k_{0}}%
}+\cdots+\dfrac{1}{p_{m}}\right)  +\frac{m-k_{0}-1}{2}+\frac{1}{2-\varepsilon
_{i}}\\
&  >\frac{\left(  m-k_{0}+1\right)  +1}{2}-\left(  \frac{1}{p_{k_{0}}}%
+\cdots+\frac{1}{p_{m}}\right)  \text{.}%
\end{align*}
and it contradicts (\ref{colombia}).

\section{Proof of Theorem \ref{vector}}

In this section, for all $r\geq1$, we denote
\[
\frac{1}{\lambda_{r}}:=\frac{1}{r}-\left(  \frac{1}{p_{1}}+\cdots+\frac
{1}{p_{m}}\right)  \text{.}%
\]

Vector-valued Hardy--Littlewood inequalities are in general associated to the
theory of absolutely summing operators, as it becomes clear in the following result:

\begin{proposition}
[{See \cite[Proposition 3.1]{dimant}}]\label{dimant} Let $E$ be Banach space,
$F$ be a cotype $q$ Banach space and $v\colon E\rightarrow F$ be an absolutely
$(r,1)$--summing operator (with $1\leq r\leq q$). If $p_{1},\ldots,p_{m}%
\in\left[  1,\infty\right]  $ and $1/p_{1}+\cdots+1/p_{m}\leq1/r-1/q$, then
there is a constant $C_{m}$ such that
\[
\left(  \sum_{j_{i}=1}^{n}\left(  \sum_{\widehat{j_{i}}=1}^{n}\Vert
vA(e_{j_{1}},\dots,e_{j_{m}})\Vert_{F}^{q}\right)  ^{\frac{\lambda_{r}}{q}%
}\right)  ^{\frac{1}{\lambda_{r}}}\leq C_{m}\Vert A\Vert\pi_{r,1}(v)
\]
for every $m$-linear operator $A\colon\ell_{p_{1}}^{n}\times\dots\times
\ell_{p_{m}}^{n}\rightarrow E$ and all $i=1,\ldots,m$, where $\widehat{j_{i}}$
means that the sum is over all coordinates except the $i$-th coordinate.
\end{proposition}

In \cite[Theorem 2.2]{alb2}, Proposition \ref{dimant} was extended in the
following fashion:

\begin{theorem}
$($See \cite[Theorem 2.2]{alb2}$)$\label{israel} Let $E$ be Banach space, $F$
be Banach space with cotype $q$ and $1\leq r\leq q$. If $p_{1},\ldots,p_{m}%
\in\left[  1,\infty\right]  $ are such that $1/p_{1}+\cdots+1/p_{m}<1/r$ and
$t_{1},\ldots,t_{m}\in\left[  \lambda_{r},\max\left\{  \lambda_{r},q\right\}
\right]  $ are such that%
\[
\frac{1}{t_{1}}+\cdots+\frac{1}{t_{m}}\leq\frac{1}{\lambda_{r}}+\frac
{m-1}{\max\left\{  \lambda_{r},q\right\}  }\text{,}%
\]
then there is a constant $C_{m}$ such that%
\begin{equation}
\left(  \sum_{j_{1}=1}^{n}\left(  \cdots\left(  \sum_{j_{m}=1}^{n}\left\Vert
vA\left(  e_{j_{1}},\ldots,e_{j_{m}}\right)  \right\Vert _{F}^{t_{m}}\right)
^{\frac{t_{m}-1}{t_{m}}}\cdots\right)  ^{\frac{t_{1}}{t_{2}}}\right)
^{\frac{1}{t_{1}}}\leq C_{m}\left\Vert A\right\Vert \pi_{r,1}\left(  v\right)
\text{,} \label{abps}%
\end{equation}
for all $m$-linear operators $A\colon\ell_{p_{1}}^{n}\times\cdots\times
\ell_{p_{m}}^{n}\rightarrow E$ and every absolutely $\left(  r,1\right)
$-summing operator $v\colon E\rightarrow F$.
\end{theorem}

Observe that when $\max\left\{  \lambda_{r},q\right\}  =\lambda_{r}$, all
exponents $t_{1},\ldots,t_{m}$ in $(\ref{abps})$ are $\lambda_{r}$. Note also
that $\max\left\{  \lambda_{r},q\right\}  =\lambda_{r}$ combined with
$1/p_{1}+\cdots+1/p_{m}<1/r$ is equivalent to
\[
\frac{1}{r}-\frac{1}{q}\leq\frac{1}{p_{1}}+\cdots+\frac{1}{p_{m}}<\frac{1}%
{r}\text{.}%
\]
The main result of this section (Theorem \ref{vector}) improves Theorem
\ref{israel} when $\max\left\{  \lambda_{r},q\right\}  =\lambda_{r}$. We begin
with some preliminary results that shall be used later

\subsection{Preliminaries}

The proof of Theorem \ref{vector} shall invoke two results stated below. The
first one appears in \cite{AAPR} and the second one is a slight extension of
Proposition \ref{dimant}:

\begin{theorem}
\cite[Theorem 2.1]{AAPR}\label{661}\ Let $F$ be a Banach space with cotype
$\rho$. If $1/p_{1}+\cdots+1/p_{m}<1/\rho$, then there is a constant
$C_{m}\geq1$ such that%
\[
\left(  \sum_{j_{1}=1}^{n}\left(  \sum_{j_{2}=1}^{n}\cdots\left(  \sum
_{j_{m}=1}^{n}\left\Vert A(e_{j_{1}},\ldots,e_{j_{m}})\right\Vert _{F}^{r_{m}%
}\right)  ^{\frac{r_{m-1}}{r_{m}}}\cdots\right)  ^{\frac{r_{1}}{r_{2}}%
}\right)  ^{\frac{1}{r_{1}}}\leq C_{m}\left\Vert A\right\Vert
\]
for all $m$-linear operators $A\colon\ell_{p_{1}}^{n}\times\cdots\times
\ell_{p_{m}}^{n}\rightarrow F$, with%
\[
r_{k}=\left[  \dfrac{1}{\rho}-\left(  \dfrac{1}{p_{k}}+\cdots+\dfrac{1}{p_{m}%
}\right)  \right]  ^{-1}\text{,}%
\]
for all $k=1,\ldots,m$.
\end{theorem}

Now we are able to prove the aforementioned extension of Proposition
\ref{dimant}.

\begin{proposition}
\label{jfa} Let $E$ be Banach space, $F$ be a cotype $q$ Banach space and
$v\colon E\rightarrow F$ be an absolutely $(r,1)$--summing operator (with
$1\leq r\leq q$). If $p_{1},\ldots,p_{m}\in\left[  1,\infty\right]  $ and
$1/p_{1}+\cdots+1/p_{m}\leq1/r-1/q$ are such that%
\begin{equation}
\sum_{k\neq i}\frac{1}{p_{k}}\leq\frac{1}{r}-\frac{1}{q} \label{21a}%
\end{equation}
for some $i\in\left\{  1,\dots,m\right\}  $, then there is a constant $C_{m}$,
such that
\[
\left(  \sum_{j_{i}=1}^{n}\left(  \sum_{\widehat{j_{i}}=1}^{n}\Vert
vA(e_{j_{1}},\dots,e_{j_{m}})\Vert_{F}^{q}\right)  ^{\frac{\lambda_{r}}{q}%
}\right)  ^{\frac{1}{\lambda_{r}}}\leq C_{m}\Vert A\Vert\pi_{r,1}(v)
\]
for every $m$-linear operator $A\colon\ell_{p_{1}}^{n}\times\dots\times
\ell_{p_{m}}^{n}\rightarrow E$.
\end{proposition}

\begin{proof}
Let $A\colon\ell_{p_{1}}^{n}\times\dots\times\ell_{p_{m}}^{n}\rightarrow E$ be
an $m$-linear operator. Choose an index $i$ satisfying (\ref{21a}) and fix
$x\in\ell_{p_{i}}^{n}$. Consider
\[
A_{i}\colon\ell_{p_{1}}^{n}\times\cdots\times\ell_{p_{i-1}}^{n}\times
\ell_{\infty}^{n}\times\ell_{p_{i+1}}^{n}\times\cdots\times\ell_{p_{m}}%
^{n}\rightarrow E
\]
defined by%
\[
A_{i}(z^{(1)},\dots,z^{(m)})=A(z^{(1)},\dots,z^{\left(  i-1\right)  }%
,xz^{(i)},z^{\left(  i+1\right)  },\dots,z^{(m)})\text{,}%
\]
where $xz^{(i)}=\left(  x_{j}z_{j}^{(i)}\right)  _{j=1}^{n}$. Let
$1/\lambda_{r}^{\prime}:=1/\lambda_{r}+1/p_{i}$. By applying Proposition
\ref{dimant} to $A_{i}$, we know that%
\begin{align}
&  \left(  \sum_{j_{i}=1}^{n}\left\vert x_{j_{i}}\right\vert ^{\lambda
_{r}^{\prime}}\left(  \sum_{\widehat{j_{i}}=1}^{n}\left\Vert vA(e_{j_{1}%
},\dots,e_{j_{i-1}},e_{j_{i}},e_{j_{i+1}},\dots,e_{j_{m}})\right\Vert _{F}%
^{q}\right)  ^{\frac{1}{q}\times\lambda_{r}^{\prime}}\right)  ^{\frac
{1}{\lambda_{r}^{\prime}}}\label{1czx}\\
&  =\left(  \sum_{j_{i}=1}^{n}\left(  \sum_{\widehat{j_{i}}=1}^{n}\left\Vert
vA(e_{j_{1}},\dots,e_{j_{i-1}},x_{j_{i}}e_{j_{i}},e_{j_{i+1}},\dots,e_{j_{m}%
})\right\Vert _{F}^{q}\right)  ^{\frac{1}{q}\times\lambda_{r}^{\prime}%
}\right)  ^{\frac{1}{\lambda_{r}^{\prime}}}\nonumber\\
&  =\left(  \sum_{j_{i}=1}^{n}\left(  \sum_{\widehat{j_{i}}=1}^{n}\left\Vert
vA_{i}(e_{j_{1}},\dots,e_{j_{m}})\right\Vert _{F}^{q}\right)  ^{\frac{1}%
{q}\times\lambda_{r}^{\prime}}\right)  ^{\frac{1}{\lambda_{r}^{\prime}}%
}\nonumber\\
&  \leq C_{m}\Vert A\Vert\Vert x\Vert_{\ell_{p_{i}}^{n}}\pi_{r,1}%
(v)\text{.}\nonumber
\end{align}
Since $\left(  p_{i}/\lambda_{r}^{\prime}\right)  ^{\ast}=\lambda_{r}%
/\lambda_{r}^{\prime}$, we get%
\begin{align*}
&  \left(  \sum_{j_{i}=1}^{n}\left(  \sum_{\widehat{j_{i}}=1}^{n}\left\Vert
vA(e_{j_{1}},\dots,e_{j_{m}})\right\Vert _{F}^{q}\right)  ^{\frac{1}{q}%
\times\lambda_{r}}\right)  ^{\frac{1}{\lambda_{r}}}\\
&  =\left(  \sum_{j_{i}=1}^{n}\left(  \sum_{\widehat{j_{i}}=1}^{n}\left\Vert
vA(e_{j_{1}},\dots,e_{j_{m}})\right\Vert _{F}^{q}\right)  ^{\frac{1}{q}%
\times\lambda_{r}^{\prime}\times\left(  \frac{p_{i}}{\lambda_{r}^{\prime}%
}\right)  ^{\ast}}\right)  ^{\frac{1}{\lambda_{r}^{\prime}}\times\frac
{1}{\left(  \frac{p_{i}}{\lambda_{r}^{\prime}}\right)  ^{\ast}}}\\
&  =\left(  \left\Vert \left(  \left(  \sum_{\widehat{j_{i}}=1}^{n}\left\Vert
vA(e_{j_{1}},\dots,e_{j_{m}})\right\Vert _{F}^{q}\right)  ^{\frac{1}{q}%
\times\lambda_{r}^{\prime}}\right)  _{j_{i}=1}^{n}\right\Vert _{\left(
\frac{p_{i}}{\lambda_{r}^{\prime}}\right)  ^{\ast}}\right)  ^{\frac{1}%
{\lambda_{r}^{\prime}}}\\
&  =\left(  \sup_{y\in B_{\ell_{\frac{p_{i}}{\lambda_{r}^{\prime}}}^{n}}}%
\sum_{j_{i}=1}^{n}\left\vert y_{j_{i}}\right\vert \left(  \sum_{\widehat
{j_{i}}=1}^{n}\left\Vert vA(e_{j_{1}},\dots,e_{j_{m}})\right\Vert _{F}%
^{q}\right)  ^{\frac{1}{q}\times\lambda_{r}^{\prime}}\right)  ^{\frac
{1}{\lambda_{r}^{\prime}}}\\
&  =\left(  \sup_{x\in B_{\ell_{p_{i}}^{n}}}\sum_{j_{i}=1}^{n}\left\vert
x_{j_{i}}\right\vert ^{\lambda_{r}^{\prime}}\left(  \sum_{\widehat{j_{i}}%
=1}^{n}\left\Vert vA(e_{j_{1}},\dots,e_{j_{m}})\right\Vert _{F}^{q}\right)
^{\frac{1}{q}\times\lambda_{r}^{\prime}}\right)  ^{\frac{1}{\lambda
_{r}^{\prime}}}\leq C_{m}\Vert A\Vert\pi_{r,1}(v)\text{,}%
\end{align*}
where the last inequality holds by (\ref{1czx}).
\end{proof}

As a corollary, we obtain a cotype $q$ version of \cite[Proposition 4.1]{alb},
which is of independent interest:

\begin{corollary}
Let $E,F$ be Banach spaces where $F$ has cotype $q$. Let $v\colon E\rightarrow
F$ be an absolutely $(r,1)$-summing operator with $1\leq r\leq q$. If
$p_{1},\ldots,p_{m}\in\left[  1,\infty\right]  $ and $1/p_{1}+\cdots
+1/p_{m}\leq1/r-1/q$ are such that $\ $%
\[
\sum_{k\neq i}\frac{1}{p_{k}}\leq\frac{1}{r}-\frac{1}{q}%
\]
for all $i\in\{1,\dots,m\}$, then there is a constant $C_{m}$, such that
\[
\left(  \sum_{j_{i}=1}^{n}\left(  \sum_{\widehat{j_{i}}=1}^{n}\Vert
vA(e_{j_{1}},\dots,e_{j_{m}})\Vert_{F}^{q}\right)  ^{\frac{\lambda_{r}}{q}%
}\right)  ^{\frac{1}{\lambda_{r}}}\leq C_{m}\Vert A\Vert\pi_{r,1}(v)
\]
for every $m$-linear operator $A\colon\ell_{p_{1}}^{n}\times\dots\times
\ell_{p_{m}}^{n}\rightarrow E$ for all $i\in\{1,\dots,m\}$.
\end{corollary}

\subsection{Proof of Theorem \ref{vector}}

Observe that, if $k_{0}=1$, we have $\frac{1}{p_{2}}+\cdots+\frac{1}{p_{m}%
}<\frac{1}{r}-\frac{1}{q}$ and taking
\[
s_{1}=\left[  \dfrac{1}{r}-\left(  \dfrac{1}{p_{1}}+\cdots+\dfrac{1}{p_{m}%
}\right)  \right]  ^{-1}%
\]
and $s_{k}=q$ for all $k>1$, in this case, the result is proved by Proposition
\ref{jfa}. Let us consider $k_{0}>1$. Let $A\colon\ell_{p_{1}}^{n}\times
\cdots\times\ell_{p_{m}}^{n}\rightarrow E$ be an $m$-linear operator and
$v\colon E\rightarrow F$ be an absolutely $\left(  r,1\right)  $-summing
operator. By%
\begin{equation}
\frac{1}{r}-\frac{1}{q}\leq\dfrac{1}{p_{k_{0}}}+\cdots+\dfrac{1}{p_{m}}
\label{14dezaa}%
\end{equation}
we have
\begin{equation}
\frac{1}{\rho}:=\dfrac{1}{r}-\left(  \dfrac{1}{p_{k_{0}}}+\cdots+\dfrac
{1}{p_{m}}\right)  \leq\frac{1}{q}\text{.} \label{Coti}%
\end{equation}
Since%
\[
\frac{1}{p_{1}}+\cdots+\frac{1}{p_{m}}<\frac{1}{r}\text{,}%
\]
we have%
\begin{equation}
\frac{1}{p_{1}}+\cdots+\frac{1}{p_{k_{0}-1}}<\frac{1}{\rho} \label{set113}%
\end{equation}
and, by definition of $k_{0}$, we obtain%
\begin{equation}
\dfrac{1}{p_{k_{0}+1}}+\cdots+\dfrac{1}{p_{m}}<\dfrac{1}{r}-\dfrac{1}%
{q}\text{.} \label{isra}%
\end{equation}
By (\ref{Coti}) we have $\rho\geq q$ and thus $\ell_{\rho}^{n}\left(  \ell
_{q}^{n}\cdots\left(  \ell_{q}^{n}\left(  F\right)  \cdots\right)  \right)  $
has cotype $\rho$. Define the $\left(  k_{0}-1\right)  $-linear operator%
\[
vA_{e}\colon\ell_{p_{1}}^{n}\times\cdots\times\ell_{p_{k_{0}-1}}%
^{n}\rightarrow\ell_{\rho}^{n}\left(  \ell_{q}^{n}\cdots\left(  \ell_{q}%
^{n}\left(  F\right)  \cdots\right)  \right)
\]
by%
\[
vA_{e}(x^{\left(  1\right)  },\ldots,x^{\left(  k_{0}-1\right)  })=\left(
vA\left(  x^{\left(  1\right)  },\ldots,x^{\left(  k_{0}-1\right)
},e_{j_{k_{0}}},\ldots,e_{j_{m}}\right)  \right)  _{j_{k_{0}},\ldots,j_{m}%
=1}^{n}\text{.}%
\]
Note that, there is a constant $C_{m-k_{0}+1}$ such that
\[
\left\Vert vA_{e}\right\Vert \leq C_{m-k_{0}+1}\left\Vert A\right\Vert
\pi_{r,1}\left(  v\right)  \text{.}%
\]
In fact, for fixed $x^{(1)},\ldots,x^{(k_{0}-1)}$, using (\ref{14dezaa}) and
(\ref{isra}), by Proposition \ref{jfa} we obtain (for the respective $\left(
m-k_{0}+1\right)  -$linear operator)
\begin{align*}
&  \left(  \sum_{j_{k_{0}}=1}^{n}\left(  \sum_{j_{k_{0}+1},\ldots,j_{m}=1}%
^{n}\Vert vA\left(  x^{\left(  1\right)  },\ldots,x^{\left(  k_{0}-1\right)
},e_{j_{k_{0}}},\ldots,e_{j_{m}}\right)  \Vert_{F}^{q}\right)  ^{\frac{\rho
}{q}}\right)  ^{\frac{1}{\rho}}\\
&  \leq C_{m-k_{0}+1}\pi_{r,1}(v)\left\Vert A\left(  x^{\left(  1\right)
},\ldots,x^{\left(  k_{0}-1\right)  },\cdot,\ldots,\cdot\right)  \right\Vert
\\
&  \leq C_{m-k_{0}+1}\pi_{r,1}(v)\left\Vert A\right\Vert \left\Vert x^{\left(
1\right)  }\right\Vert \cdot\cdots\cdot\left\Vert x^{\left(  k_{0}-1\right)
}\right\Vert \text{.}%
\end{align*}
Thus
\begin{align*}
&  \left\Vert vA_{e}\right\Vert =\sup_{\left\Vert x^{(1)}\right\Vert
,\ldots,\left\Vert x^{(k_{0}-1)}\right\Vert \leq1}\left\Vert vA_{e}\left(
x^{\left(  1\right)  },\ldots,x^{\left(  k_{0}-1\right)  }\right)  \right\Vert
_{\ell_{\rho}^{n}\left(  \ell_{q}^{n}\cdots\left(  \ell_{q}^{n}\left(
F\right)  \cdots\right)  \right)  }\\
&  =\sup_{\left\Vert x^{\left(  1\right)  }\right\Vert ,\ldots,\left\Vert
x^{\left(  k_{0}-1\right)  }\right\Vert \leq1}\left(  \sum_{j_{k_{0}}=1}%
^{n}\left(  \sum_{j_{k_{0}+1},\ldots,j_{m}=1}^{n}\Vert vA\left(  x^{\left(
1\right)  },\ldots,x^{\left(  k_{0}-1\right)  },e_{j_{k_{0}}},\ldots,e_{j_{m}%
}\right)  \Vert_{F}^{q}\right)  ^{\frac{\rho}{q}}\right)  ^{\frac{1}{\rho}}\\
&  \leq C_{m-k_{0}+1}\left\Vert A\right\Vert \pi_{r,1}\left(  v\right)
\text{,}%
\end{align*}
as required.

On the other hand, since $\ell_{\rho}^{n}\left(  \ell_{q}^{n}\cdots\left(
\ell_{q}^{n}\left(  F\right)  \cdots\right)  \right)  $ has cotype $\rho$ and
using (\ref{set113}) and Theorem \ref{661} for the $\left(  k_{0}-1\right)
$-linear operator $vA_{e}$, we conclude that there is a constant $C_{k_{0}-1}$
such that
\begin{align*}
&  \left(  \sum_{j_{1}=1}^{n}\left(  \sum_{j_{2}=1}^{n}\cdots\left(
\sum_{j_{k_{0}-1}=1}^{n}\left(  \sum_{j_{k_{0}}=1}^{n}\left(  \sum
_{j_{k_{0}+1},\ldots,j_{m}=1}^{n}\left\Vert vA(e_{j_{1}},\ldots,e_{j_{m}%
})\right\Vert _{F}^{q}\right)  ^{\frac{\rho}{q}}\right)  ^{\frac{r_{k_{0}-1}%
}{\rho}}\right)  ^{\frac{r_{k_{0}-2}}{r_{k_{0}-1}}}\cdots\right)
^{\frac{r_{1}}{r_{2}}}\right)  ^{\frac{1}{r_{1}}}\\
&  =\left(  \sum_{j_{1}=1}^{n}\left(  \sum_{j_{2}=1}^{n}\cdots\left(
\sum_{j_{k_{0}-1}=1}^{n}\left\Vert vA_{e}(e_{j_{1}},\ldots,e_{j_{k_{0}-1}%
})\right\Vert _{\ell_{\rho}^{n}\left(  \ell_{q}^{n}\cdots\left(  \ell_{q}%
^{n}\left(  F\right)  \cdots\right)  \right)  }^{r_{k_{0}-1}}\right)
^{\frac{r_{k_{0}-2}}{r_{k_{0}-1}}}\cdots\right)  ^{\frac{r_{1}}{r_{2}}%
}\right)  ^{\frac{1}{r_{1}}}\\
&  \leq C_{k_{0}-1}\left\Vert vA_{e}\right\Vert \\
&  \leq C_{k_{0}-1}C_{m-k_{0}+1}\left\Vert A\right\Vert \pi_{r,1}\left(
v\right)  \text{,}%
\end{align*}
with%
\[
r_{k}=%
\begin{array}
[c]{ll}%
\left[  \dfrac{1}{\rho}-\left(  \dfrac{1}{p_{k}}+\cdots+\dfrac{1}{p_{k_{0}-1}%
}\right)  \right]  ^{-1}\text{,} & \text{if }k\leq k_{0}-1\text{.}%
\end{array}
\]
Note that for $k=k_{0}$ we have $s_{k}=\rho$ and for $k>k_{0}$ we have
$q=s_{k}$. Finally, for $k<k_{0}$ we have $r_{k}=s_{k}$, since
\[
\frac{1}{r_{k}}=\dfrac{1}{\rho}-\left(  \dfrac{1}{p_{k}}+\cdots+\dfrac
{1}{p_{k_{0}-1}}\right)  =\dfrac{1}{r}-\left(  \dfrac{1}{p_{k_{0}}}%
+\cdots+\dfrac{1}{p_{m}}\right)  -\left(  \dfrac{1}{p_{k}}+\cdots+\dfrac
{1}{p_{k_{0}-1}}\right)  =\dfrac{1}{r}-\left(  \dfrac{1}{p_{k}}+\cdots
+\dfrac{1}{p_{m}}\right)  \text{,}%
\]
the proof is done.

In order to illustrate how Theorem \ref{vector} improves (numerically) Theorem
\ref{israel}, let us consider $E=\ell_{1}$, and $F=\ell_{2}$, $m=9$, and
$p_{1}=\cdots=p_{9}=10$. In this case we have $q=2$ and by Grothendieck's
theorem we can choose $r=1$ for all $v\colon\ell_{1}\rightarrow\ell_{2}$.
Since
\[
1-\frac{1}{2}\leq\frac{1}{p_{1}}+\cdots+\frac{1}{p_{9}}=\frac{9}{10}<1\text{,}%
\]
we have $k_{0}=5$ and we obtain the table below which compares the exponents
provided by Theorems \ref{vector} and Theorem \ref{israel} for $\left(
E,F,m,r,p_{j}\right)  =\left(  \ell_{1},\ell_{2},9,1,10\right)  $ for all
$j=1,\ldots,9$:%

%TCIMACRO{\TeXButton{Tabela2}{\begin{table}[H]
%\centering\caption{}$
%\begin{tabular}
%[c]{|c|c|c|c|c|c|c|c|c|c|}\hline& $s_{1}$ & $s_{2}$ & $s_{3}$ & $s_{4}%
%$ & $s_{5}$ & $s_{6}$ & $s_{7}$ &
%$s_{8}$ & $s_{9}$\\\hline\multicolumn{1}{|c|}{\cite[Theorem 2.2]{alb2}}
%& $10$ & $10$ & $10$ & $10$ &
%$10$ & $10$ & $10$ & $10$ & $10$\\\hline\multicolumn{1}{|c|}{Theorem \ref
%{vector}} & $10$ & $5$ & $\simeq3.3$ & $2.5$
%& $2$ & $2$ & $2$ & $2$ & $2$\\\hline\end{tabular}
%\ $\end{table}}}%
%BeginExpansion
\begin{table}[H]
\centering\caption{}$
\begin{tabular}
[c]{|c|c|c|c|c|c|c|c|c|c|}\hline& $s_{1}$ & $s_{2}$ & $s_{3}$ & $s_{4}%
$ & $s_{5}$ & $s_{6}$ & $s_{7}$ &
$s_{8}$ & $s_{9}$\\\hline\multicolumn{1}{|c|}{\cite[Theorem 2.2]{alb2}}
& $10$ & $10$ & $10$ & $10$ &
$10$ & $10$ & $10$ & $10$ & $10$\\\hline\multicolumn{1}{|c|}{Theorem \ref
{vector}} & $10$ & $5$ & $\simeq3.3$ & $2.5$
& $2$ & $2$ & $2$ & $2$ & $2$\\\hline\end{tabular}
\ $\end{table}%
%EndExpansion

\section{The critical case: globally sharp exponents}

Until very recently, the HL inequalities were just investigated for
$1/p_{1}+\cdots+1/p_{m}<1$. The reason was very simple: if we consider
$1/p_{1}+\cdots+1/p_{m}\geq1$, there does not exist a finite exponent $s$ for
which there is a constant $C_{m}$ satisfying%
\[
\left(  \sum_{j_{1}=1}^{n}\cdots\sum_{j_{m}=1}^{n}\left\vert A\left(
e_{j_{1}},\ldots,e_{j_{m}}\right)  \right\vert ^{s}\right)  ^{\frac{1}{s}}\leq
C_{m}\left\Vert A\right\Vert \text{,}%
\]
for all $m$-linear forms $A\colon\ell_{p_{1}}^{n}\times\cdots\times\ell
_{p_{m}}^{n}\rightarrow\mathbb{K}$. So, at first glance, it seemed that no
theory was supposed to be expected in this framework. However, this is a
blurred perspective; when we consider just one exponent $s$ at all sums, we
lose information. So, in \cite{anais, paulino}, the authors initiated the
investigation of the case $1/p_{1}+\cdots+1/p_{m}\geq1$ under an anisotropic
viewpoint. In this section we follow this vein and obtain globally sharp
exponents for the case $1/p_{1}+\cdots+1/p_{m}=1$. Hereafter, for the sake of
simplicity, when $s=\infty$, the sum $\left(  \sum_{j=1}^{\infty}\left\vert
a_{j}\right\vert ^{s}\right)  ^{1/s}$ denotes $\sup\left\vert a_{j}\right\vert
$.

We recall that in this case some exponents in the anisotropic Hardy-Littlewood
inequality are forced to be infinity. The first result dealing with this
notion is the following:

\begin{theorem}
\label{criticodjair}$($See \cite[Theorem 1]{paulino}$)$ There is a constant
$C_{m}$ such that
\begin{equation}
\sup_{j_{1}}\left(  \sum_{j_{2}=1}^{n}\left(  \cdots\left(  \sum_{j_{m}=1}%
^{n}\left\vert A\left(  e_{j_{1}},\dots,e_{j_{m}}\right)  \right\vert ^{s_{m}%
}\right)  ^{\frac{1}{s_{m}}s_{m-1}}\cdots\right)  ^{\frac{1}{s_{3}}s_{2}%
}\right)  ^{\frac{1}{s_{2}}}\leq C_{m}\left\Vert A\right\Vert \label{91}%
\end{equation}
for all $m$-linear forms $A\colon\ell_{m}^{n}\times\cdots\times\ell_{m}%
^{n}\rightarrow\mathbb{K}$, and all positive integers $n$, with
\[
s_{k}=\frac{2m(m-1)}{mk-2k+2}%
\]
for all $k=2,\ldots,m$. Moreover, $s_{1}=\infty$ and $s_{2}=m$ are sharp and,
for $m>2$ the optimal exponents $s_{k}$ satisfying $(\ref{91})$ fulfill
\[
s_{k}\geq\frac{m}{k-1}\text{.}%
\]

\end{theorem}

As a consequence of Theorem \ref{9876}, we have the following generalization
of the previous theorem:

\begin{proposition}
\label{9876crit}Let $p_{1},\ldots,p_{m}\in\left[  1,\infty\right]  $ be such
that $1/2\leq1/p_{2}+\cdots+1/p_{m}<1$ and%
\begin{equation}
\frac{1}{p_{1}}+\cdots+\frac{1}{p_{m}}=1\text{.} \label{hip33}%
\end{equation}
There is a constant $C_{m}$ such that
\begin{equation}
\left(  \sum_{j_{1}=1}^{n}\left(  \cdots\left(  \sum_{j_{m}=1}^{n}\left\vert
A\left(  e_{j_{1}},\ldots,e_{j_{m}}\right)  \right\vert ^{s_{m}}\right)
^{\frac{s_{m-1}}{s_{m}}}\cdots\right)  ^{\frac{s_{1}}{s_{2}}}\right)
^{\frac{1}{s_{1}}}\leq C_{m}\left\Vert A\right\Vert \label{2222}%
\end{equation}
for all $m$-linear forms $A\colon\ell_{p_{1}}^{n}\times\cdots\times\ell
_{p_{m}}^{n}\rightarrow\mathbb{K}$, where%
\[
s_{k}=\left\{
\begin{array}
[c]{ll}%
\left[  1-\left(  \dfrac{1}{p_{k}}+\cdots+\dfrac{1}{p_{m}}\right)  \right]
^{-1}\text{,} & \text{if }1\leq k\leq k_{0}:=\max\left\{  t:\dfrac{1}{p_{t}%
}+\cdots+\dfrac{1}{p_{m}}\geq\dfrac{1}{2}\right\}  \text{,}\\
\vspace{-0.3cm} & \\
2\text{,} & \text{if }k>k_{0}\text{.}%
\end{array}
\right.
\]
Moreover:

\begin{enumerate}
\item[\emph{(i)}] The exponents $s_{1},\ldots,s_{k_{0}}$ are optimal.

\item[\emph{(ii)}] If $p_{k_{0}}\geq2$, all the exponents are globally sharp.
\end{enumerate}
\end{proposition}

\begin{proof}
The proof of the existence is a simple consequence of Theorem \ref{9876}; in
fact, by Theorem \ref{9876}, for any fixed vector $e_{j_{1}}$, there is a
constant $C_{m}$ such that
\[
\left(  \sum_{j_{2}=1}^{n}\left(  \cdots\left(  \sum_{j_{m}=1}^{n}\left\vert
A\left(  e_{j_{1}},\ldots,e_{j_{m}}\right)  \right\vert ^{s_{m}}\right)
^{\frac{s_{m-1}}{s_{m}}}\cdots\right)  ^{\frac{s_{2}}{s_{3}}}\right)
^{\frac{1}{s_{2}}}\leq C_{m}\left\Vert A\right\Vert
\]
for all $m$-linear forms $A\colon\ell_{p_{1}}^{n}\times\cdots\times\ell
_{p_{m}}^{n}\rightarrow\mathbb{K}$ and this easily implies (\ref{2222}).

Now we shall prove (i) and (ii).

In order to prove (i), note that if $s_{1}=\infty$ could be improved, there
would exist $r\in\left(  0,\infty\right)  $ and $C_{m}$ such that%
\[
\left(  \sum_{j_{1}=1}^{n}\left(  \sum_{j_{2}=1}^{n}\left(  \cdots\left(
\sum_{j_{m}=1}^{n}\left\vert A\left(  e_{j_{1}},\ldots,e_{j_{m}}\right)
\right\vert ^{s_{m}}\right)  ^{\frac{s_{m-1}}{s_{m}}}\cdots\right)
^{\frac{s_{2}}{s_{3}}}\right)  ^{\frac{r}{s_{2}}}\right)  ^{\frac{1}{r}}\leq
C_{m}\left\Vert A\right\Vert
\]
for all $m$-linear forms $A\colon\ell_{p_{_{1}}}^{n}\times\cdots\times
\ell_{p_{m}}^{n}\rightarrow\mathbb{K}$. Considering $\rho=\max\left\{
s_{2},\ldots,s_{m},r\right\}  $, by the monotonicity of the $\ell_{q}$ norms
we would conclude that
\[
\left(  \sum_{j_{1},\ldots,j_{m}=1}^{n}\left\vert A(e_{j_{1}},\ldots,e_{j_{m}%
})\right\vert ^{\rho}\right)  ^{\frac{1}{\rho}}\leq C_{m}\left\Vert
A\right\Vert
\]
for all $m$-linear forms $A\colon\ell_{p_{1}}^{n}\times\cdots\times\ell
_{p_{m}}^{n}\rightarrow\mathbb{K}$, but this is impossible due to (\ref{hip33}).

On the other hand, note that if
\[
\sup_{j_{1}}\left(  \sum_{j_{2}=1}^{n}\left(  \cdots\left(  \sum_{j_{m}=1}%
^{n}\left\vert A\left(  e_{j_{1}},\ldots,e_{j_{m}}\right)  \right\vert
^{s_{m}}\right)  ^{\frac{s_{m-1}}{s_{m}}}\cdots\right)  ^{\frac{s_{2}}{s_{3}}%
}\right)  ^{\frac{1}{s_{2}}}\leq C_{m}\left\Vert A\right\Vert \text{,}%
\]
for all $m$-linear forms $A\colon\ell_{p_{1}}^{n}\times\cdots\times\ell
_{p_{m}}^{n}\rightarrow\mathbb{K}$, then by Lemma \ref{8u8u} we have%
\[
\left(  \sum_{j_{2}=1}^{n}\left(  \cdots\left(  \sum_{j_{m}=1}^{n}\left\vert
A\left(  e_{j_{2}},\ldots,e_{j_{m}}\right)  \right\vert ^{s_{m}}\right)
^{\frac{s_{m-1}}{s_{m}}}\cdots\right)  ^{\frac{s_{2}}{s_{3}}}\right)
^{\frac{1}{s_{2}}}\leq C_{m}\left\Vert A\right\Vert \text{,}%
\]
for all $\left(  m-1\right)  $-linear forms $A\colon\ell_{p_{2}}^{n}%
\times\cdots\times\ell_{p_{m}}^{n}\rightarrow\mathbb{K}$. Hence, the proofs of
(i) and (ii) are completed as consequence of Theorem \ref{9876}.
\end{proof}

As a consequence, we have an improvement of Theorem \ref{criticodjair}:

\begin{corollary}
\label{corocrit} There exists a constant $C_{m}$ such that
\begin{equation}
\left(  \sum_{j_{1}=1}^{n}\left(  \cdots\left(  \sum_{j_{m}=1}^{n}\left\vert
A\left(  e_{j_{1}},\dots,e_{j_{m}}\right)  \right\vert ^{s_{m}}\right)
^{\frac{1}{s_{m}}s_{m-1}}\cdots\right)  ^{\frac{1}{s_{2}}s_{1}}\right)
^{\frac{1}{s_{1}}}\leq C_{m}\left\Vert A\right\Vert \label{w33}%
\end{equation}
for all $m$-linear forms $A\colon\ell_{m}^{n}\times\cdots\times\ell_{m}%
^{n}\rightarrow\mathbb{K}$, with%
\[
s_{k}=\left\{
\begin{array}
[c]{ll}%
\dfrac{m}{k-1}\text{,} & \text{if }1<k\leq k_{0}\\
\vspace{-0.3cm} & \\
2\text{,} & \text{if }k>k_{0}\text{,}%
\end{array}
\right.
\]
where $k_{0}:=\left\lfloor \frac{m+2}{2}\right\rfloor $ $($the largest integer
less than or equal to $\left(  m+2\right)  /2)$. Moreover, $s_{1}%
,\ldots,s_{k_{0}}$ are sharp, and $\left(  s_{1},\ldots,s_{m}\right)  $ is
globally sharp.
\end{corollary}

\begin{proof}
We shall use the previous proposition with $p_{1}=\cdots=p_{m}=m$. Observe
that, if $m=2N+1$ or $m=2N$, then $\left\lfloor \frac{m+2}{2}\right\rfloor
=N+1$. Let $k_{0}$ be as in the previous theorem, i.e.,
\[
k_{0}:=\max\left\{  t:\dfrac{1}{p_{t}}+\cdots+\dfrac{1}{p_{m}}\geq\dfrac{1}%
{2}\right\}  \text{.}%
\]
Since%
\[
\frac{1}{p_{t}}+\cdots+\frac{1}{p_{m}}\geq\frac{1}{2}\Leftrightarrow
t\leq\frac{m+2}{2}\text{,}%
\]
we conclude that $k_{0}=\left\lfloor \frac{m+2}{2}\right\rfloor $. Hence, by
Proposition \ref{9876crit} we conclude that (\ref{w33}) holds as well as the
optimality of the exponents.
\end{proof}

In order to illustrate how Proposition \ref{9876crit} and its corollary
improve \cite[Theorem 1]{paulino}, let us consider $m=10$ in the table below:%

%TCIMACRO{\TeXButton{Tabela3}{\begin{table}[H]
%\centering\caption{}$
%\begin{tabular}
%[c]{|c|c|c|c|c|c|c|c|c|c|c|}\hline& $s_{1}$ & $s_{2}$ & $s_{3}$ & $s_{4}
%$ & $s_{5}$ & $s_{6}$ & $s_{7}$ &
%$s_{8}$ & $s_{9}$ & $s_{10}$\\\hline\multicolumn{1}{|c|}{\cite
%[Theorem 1]{paulino}} & $\infty$ & $10$ &
%$\simeq6.92$ & $\simeq5.29$ & $\simeq4.28$ & $3.6$ & $\simeq3.10$ &
%$\simeq2.72$ & $\simeq2.43$ & $\simeq2.19$\\\hline\multicolumn{1}%
%{|c|}{Corollary \ref{corocrit}} & \color{red}$\infty$ & \color{red}
%$10$ & \color{red}$5$ &
%\color{red}$\simeq3.33$ & \color{red}$2.5$ & \color{red}
%$2$ & \color{blue}$2$ & \color{blue}$2$ & \color{blue}$2$ & \color
%{blue}$2$\\\hline\end{tabular}
%\ $\\
%\begin{flushleft}
%{\color{red} $\blacksquare$} The exponents are sharp\\
%{\color{blue} $\blacksquare$}
%The exponents (combined with the exponents in red) are globally sharp
%\end{flushleft}
%\end{table}}}%
%BeginExpansion
\begin{table}[H]
\centering\caption{}$
\begin{tabular}
[c]{|c|c|c|c|c|c|c|c|c|c|c|}\hline& $s_{1}$ & $s_{2}$ & $s_{3}$ & $s_{4}
$ & $s_{5}$ & $s_{6}$ & $s_{7}$ &
$s_{8}$ & $s_{9}$ & $s_{10}$\\\hline\multicolumn{1}{|c|}{\cite
[Theorem 1]{paulino}} & $\infty$ & $10$ &
$\simeq6.92$ & $\simeq5.29$ & $\simeq4.28$ & $3.6$ & $\simeq3.10$ &
$\simeq2.72$ & $\simeq2.43$ & $\simeq2.19$\\\hline\multicolumn{1}%
{|c|}{Corollary \ref{corocrit}} & \color{red}$\infty$ & \color{red}
$10$ & \color{red}$5$ &
\color{red}$\simeq3.33$ & \color{red}$2.5$ & \color{red}
$2$ & \color{blue}$2$ & \color{blue}$2$ & \color{blue}$2$ & \color
{blue}$2$\\\hline\end{tabular}
\ $\\
\begin{flushleft}
{\color{red} $\blacksquare$} The exponents are sharp\\
{\color{blue} $\blacksquare$}
The exponents (combined with the exponents in red) are globally sharp
\end{flushleft}
\end{table}%
%EndExpansion

Following the lines of the proof of Proposition \ref{9876crit}, we can obtain
the following extended version of Theorem \ref{vector}:

\begin{theorem}
Let $p_{1},\ldots,p_{m}\in\left[  1,\infty\right]  $, $E$ be Banach space, $F$
be a cotype $q$ space and $1\leq r\leq q$, with%
\[
\frac{1}{p_{1}}+\cdots+\frac{1}{p_{m}}=\frac{1}{r}%
\]
and
\[
\frac{1}{r}-\frac{1}{q}\leq\frac{1}{p_{2}}+\cdots+\frac{1}{p_{m}}<\frac{1}%
{r}\text{.}%
\]
Then, there is a constant $C_{m}$ such that
\[
\left(  \sum_{j_{1}=1}^{n}\left(  \cdots\left(  \sum_{j_{m}=1}^{n}\left\Vert
vA\left(  e_{j_{1}},\ldots,e_{j_{m}}\right)  \right\Vert _{F}^{s_{m}}\right)
^{\frac{s_{m-1}}{s_{m}}}\cdots\right)  ^{\frac{s_{1}}{s_{2}}}\right)
^{\frac{1}{s_{1}}}\leq C_{m}\left\Vert A\right\Vert \pi_{r,1}\left(  v\right)
\text{,}%
\]
for all $m$-linear operators $A\colon\ell_{p_{1}}^{n}\times\cdots\times
\ell_{p_{m}}^{n}\rightarrow E$ and all absolutely $\left(  r,1\right)
$-summing operators $v\colon E\rightarrow F$, where%
\[
s_{k}=\left\{
\begin{array}
[c]{ll}%
\left[  \dfrac{1}{r}-\left(  \dfrac{1}{p_{k}}+\cdots+\dfrac{1}{p_{m}}\right)
\right]  ^{-1}\text{,} & \text{if }1\leq k\leq k_{0}:=\max\left\{  t:\dfrac
{1}{p_{t}}+\cdots+\dfrac{1}{p_{m}}\geq\dfrac{1}{r}-\dfrac{1}{q}\right\}
\text{,}\\
\vspace{-0.3cm} & \\
q\text{,} & \text{if }k>k_{0}\text{.}%
\end{array}
\right.
\]
Moreover, $s_{1}=\infty$ cannot be improved.
\end{theorem}

\section{The Regularity Principle is sharp}

In this final section we investigate a natural question: is the (Anisotropic)
Regularity Principle sharp? Of course, for special choices of Banach spaces we
can find better inclusions, as it also happens with Inclusion Theorems (see
\cite{diestel}). The right question here is whether is it possible to obtain
better inclusions keeping the full generality of the Regularity Principle. As
we shall see, a simple consequence of what we have already proved in this
paper is that the answer is no. We shall prove, for instance, that if $r\geq
2$, the parameters $s_{1},\ldots,s_{m-1}$ in (\ref{9080}) are optimal.

Let $r\geq2$ and suppose that for some $i\in\left\{  1,\ldots,m-1\right\}  $
there is $\varepsilon>0$ such that
\begin{equation}
\Pi_{\left(  r;p_{1},\ldots,p_{m}\right)  }^{m}\left(  E_{1},\ldots
,E_{m};F\right)  \subset\Pi_{\left(  t_{1},\ldots,t_{m};q_{1},\ldots
,q_{m}\right)  }^{m}\left(  E_{1},\ldots,E_{m};F\right)  \text{,}\label{fail}%
\end{equation}
for any Banach spaces $E_{1},\ldots,E_{m}$ and $F$, with $t_{i}=s_{i}%
-\varepsilon$. If we take $F=\mathbb{K}$ and $E_{i}=\ell_{p_{i}}^{n}$, with
$p_{m}=\frac{r}{r-1}\leq2$ and $\frac{1}{2}\leq\frac{1}{p_{1}}+\cdots+\frac
{1}{p_{m}}<1$, by mimicking the prove of Theorem \ref{9876} with the inclusion
(\ref{fail}) instead of (\ref{9080}), we conclude that there is a constant
$C_{m}$ such that, for all positive integers $n$ and all $m$-linear forms
$A\colon\ell_{p_{1}}^{n}\times\cdots\times\ell_{p_{m}}^{n}\rightarrow
\mathbb{K}$, we have%
\[
\left(  \sum_{j_{1}=1}^{n}\left(  \cdots\left(  \sum_{j_{m}=1}^{n}\left\vert
A\left(  e_{j_{1}},\ldots,e_{j_{m}}\right)  \right\vert ^{t_{m}}\right)
^{\frac{t_{m}-1}{t_{m}}}\cdots\right)  ^{\frac{t_{1}}{t_{2}}}\right)
^{\frac{1}{t_{1}}}\leq C\left\Vert A\right\Vert \text{,}%
\]
with $t_{i}=s_{i}-\varepsilon$. But this is impossible due to the sharpness of
$s_{i}$ in Theorem \ref{9876}. Similar arguments show that the estimate of
$\delta$ in (\ref{deltadez}) cannot be improved, in general. It is also
well-known that the hypothesis $\epsilon<p_{1}$ in Theorem \ref{RP} cannot be
relaxed to $\epsilon\leq p_{1}.$

\end{document}